\documentclass[reqno,12pt]{amsart}
\usepackage{amssymb,amsmath,amsthm,array}
\usepackage{cite}
\usepackage{mathdots}
\usepackage{dsfont}
\usepackage{graphicx}
\usepackage{float}
\usepackage{caption}
\usepackage{subcaption}
\usepackage{bm}

\def\0{{\bf 0}}
\def\1{{\bf 1}}

\def\ba{{\bm \alpha}}
\def\bb{{\bm \beta}}
\def\bg{{\bm \gamma}}
\def\bo{{\bm \omega}}

\def\A{{\bf A}}
\def\B{{\bf B}}
\def\I{{\bf I}}
\def\G{{\bf G}}
\def\M{{\bf M}}
\def\U{{\bf U}}
\def\X{{\bf X}}

\def\bl{{\bf \Lambda}}
\def\bs{{\bf \Sigma}}

\def\R{{\bf R}}
\def\RR{\mathcal{R}}
\def\RRR{\mathbb{R}}

\DeclareMathOperator{\rank}{rank}
\DeclareMathOperator{\mult}{Mult}
\DeclareMathOperator{\tr}{Tr}

\newtheorem*{remark}{Remark}
\newtheorem*{example}{Example}

\newtheorem{definition}{Definition}
\newtheorem{proposition}{Proposition}

\newcommand*{\Scale}[2][4]{\scalebox{#1}{$#2$}}%
\newcommand{\subf}[2]{%
   {\small\begin{tabular}[t]{@{}c@{}}
   #1\\#2
   \end{tabular}}%
}
\numberwithin{equation}{section}

\title{Motifs, Coherent Configurations and Second Order Network Generation }
\author{Jared C. Bronski}  
\author{Timothy Ferguson
}

\begin{document}


\begin{abstract}
In this paper we illuminate some algebraic-combinatorial structure
underlying the second order networks (SONETS)
random graph model of Nykamp, Zhao and collaborators\cite{Nykamp.Zhao.2011,Zhao.2012,Fuller.2016}. In particular we show that this algorithm is deeply connected with a certain homogeneous coherent configuration, a non-commuting generalization of the classical Johnson scheme. This algebraic structure underlies certain surprising
identities (that do not appear to have been previously observed) satisfied by the covariance matrices in the Nykamp-Zhao
scheme. We show that an understanding of this algebraic structure
leads to simplified numerical methods for carrying out the linear
algebra required to implement the SONETS algorithm. We also show that
this structure extends naturally to the problem of generating random
subgraphs of graphs other than the complete directed graph. 
\end{abstract}


\maketitle


\section{Introduction}
The Erd\H{o}s-R\'{e}nyi $G(n,p)$ model is a common model of a
random graph of $n$ nodes, where each possible edge is independently
present (resp. absent) with
probability $p$ (resp. $1-p$). However it has long been understood that many, if not
most, real world networks show substantial departures from
independence.\cite{Albert.Barabasi.2002,Milo.2002,Barabasi.2009} One indicator of the departure from independence is
given by the relative frequency of certain {\em motifs}, or subgraphs.
In real world networks it is frequently the case that the number of
certain motifs occuring in the graph differs substantially from what one would
expect based on the assumption of independent edges, indicating correlations among the edges. This has led to
a large effort to understand this phenomenon. While it is difficult to
do justice to the vast amount of literature on this subject most of
the work has proceeded in two complementary directions: developing
tools for identifying important motifs or other structures in a given (large) network\cite{Porter.Mucha.2009,Porter.Mucha.2013,Kashtan.2004,Wang.Baur.2012}, and algorithms for
generating random networks with certain desirable properties including,
but not limited to, having specified  probabilities for different
motifs. The latter random graph models include the well-known preferential attachment
model\cite{Barabasi.Albert.1999,Newman.PA.2001} and small world network models\cite{Watts.Strogatz.1998,Newman.Watts.1999}  as well as a
number of models and methods of generation\cite{Newman.2003, delGenio2010,Batagelj.2005,Diaconis.2011,Aiello.2000,Chung.2004}.  
   
The motivation for this work is the work of Nykamp, Zhao and
collaborators \cite{Nykamp.Zhao.2011,Zhao.2012,Fuller.2016} on
generating what they refer to as second order networks, or
SONETS. This  generalization of the Erd\H{o}s-R\'{e}nyi random
graph model that allows for the generation of random networks
(directed graphs) with second order correlations among the  edges:
one can not only specify the single edge probabilities, but one can
also independently vary the probabilities of all possible two edge
motifs. We show that the correlation matrices of the
SONETS model enjoy some remarkable algebraic identities. We further
show that these identities can be explained by the fact that there is
a coherent configuration, a particular type of algebraic structure,
that underlies these models. We also construct some other examples of
second order networks, and the related coherent configurations. 
In each case the existence of a coherent configuration has strong
numerical implications for the implementation of the method. In
particular all of the spectral operation that need to be done on the
covariance matrix -- mainly the extraction of the positive definite
square root -- can be done in time {\em independent of the size of the matrix}.

We begin with a brief description of the Nykamp-Zhao scheme for
generating second order networks (simple directed graphs). The Nykamp-Zhao approach
generates a random directed graph in the following manner. One first
generates a Gaussian random vector  $\bo \in \RRR^E$, where $E=N(N-1)$
is the number of potential edges. One then performs ``thresholding'':
an edge is present at site $i$ if $\omega_i > x$ and absent if
$\omega_i \leq x$ for some choice of the threshold level $x$. This would simply give an Erd\H{o}s-R\'{e}nyi
random graph if the $\omega_i$ were independent and identically
distributed (iid) -- in other words if the covariance were a
multiple of the identity matrix -- but the covariance matrix $\bs$ is instead chosen as follows. The authors first identify various two edge motifs:
\begin{itemize}
\item Reciprocal: The edges point between the same pair of vertices, in opposite directions.
\item Convergent: Both edges point inwards to a common vertex.
\item Divergent: Both edges point outward from a common vertex.
\item Chain: One edge points inward to a vertex, the other points outward from the same vertex.
\end{itemize}
We add to this collection the  disjoint motif, where the edges do not
share a common vertex. This motif was not considered in the original
SONETS work but fits into this framework naturally, and must be
included here for reasons of algebraic completeness. These motifs are illustrated in Figure \ref{fig:DJohnson}. The covariance for each pair of edges is then assigned according to the motif which they form. For instance, $\Sigma_{ij}=\alpha_{\text{div}}$ for every pair of edges $(i,j)$ that form a divergent motif. This gives an $N(N-1)\times N(N-1)$ covariance matrix $\bs(\ba)$ depending on five arbitrary constants $\ba = (\alpha_{\text{recip}},\alpha_{\text{conv}},\alpha_{\text{div}},\alpha_{\text{chain}},\alpha_{\text{disj}})$. For instance, when $N=5$ the covariance $\bs(\ba)$ is a $20\times 20$ matrix given by 
\begin{align*}
\Scale[.65]{\left(
\begin{array}{cccccccccccccccccccc}
 1 & \alpha_{\text{recip}} & \alpha_{\text{conv}} & \alpha_{\text{chain}} & \alpha_{\text{conv}} &
   \alpha_{\text{chain}} & \alpha_{\text{conv}} & \alpha_{\text{chain}} & \alpha_{\text{chain}} &
   \alpha_{\text{div}} & \alpha_{\text{chain}} & \alpha_{\text{div}} & \alpha_{\text{chain}} & \alpha_{\text{div}}
   & \alpha_{\text{disj}} & \alpha_{\text{disj}} & \alpha_{\text{disj}} & \alpha_{\text{disj}} &
   \alpha_{\text{disj}} & \alpha_{\text{disj}} \\
 \alpha_{\text{recip}} & 1 & \alpha_{\text{chain}} & \alpha_{\text{div}} & \alpha_{\text{chain}} &
   \alpha_{\text{div}} & \alpha_{\text{chain}} & \alpha_{\text{div}} & \alpha_{\text{conv}} &
   \alpha_{\text{chain}} & \alpha_{\text{conv}} & \alpha_{\text{chain}} & \alpha_{\text{conv}} &
   \alpha_{\text{chain}} & \alpha_{\text{disj}} & \alpha_{\text{disj}} & \alpha_{\text{disj}} &
   \alpha_{\text{disj}} & \alpha_{\text{disj}} & \alpha_{\text{disj}} \\
 \alpha_{\text{conv}} & \alpha_{\text{chain}} & 1 & \alpha_{\text{recip}} & \alpha_{\text{conv}} &
   \alpha_{\text{chain}} & \alpha_{\text{conv}} & \alpha_{\text{chain}} & \alpha_{\text{div}} &
   \alpha_{\text{chain}} & \alpha_{\text{disj}} & \alpha_{\text{disj}} & \alpha_{\text{disj}} &
   \alpha_{\text{disj}} & \alpha_{\text{chain}} & \alpha_{\text{div}} & \alpha_{\text{chain}} &
   \alpha_{\text{div}} & \alpha_{\text{disj}} & \alpha_{\text{disj}} \\
 \alpha_{\text{chain}} & \alpha_{\text{div}} & \alpha_{\text{recip}} & 1 & \alpha_{\text{chain}} &
   \alpha_{\text{div}} & \alpha_{\text{chain}} & \alpha_{\text{div}} & \alpha_{\text{chain}} &
   \alpha_{\text{conv}} & \alpha_{\text{disj}} & \alpha_{\text{disj}} & \alpha_{\text{disj}} &
   \alpha_{\text{disj}} & \alpha_{\text{conv}} & \alpha_{\text{chain}} & \alpha_{\text{conv}} &
   \alpha_{\text{chain}} & \alpha_{\text{disj}} & \alpha_{\text{disj}} \\
 \alpha_{\text{conv}} & \alpha_{\text{chain}} & \alpha_{\text{conv}} & \alpha_{\text{chain}} & 1 &
   \alpha_{\text{recip}} & \alpha_{\text{conv}} & \alpha_{\text{chain}} & \alpha_{\text{disj}} &
   \alpha_{\text{disj}} & \alpha_{\text{div}} & \alpha_{\text{chain}} & \alpha_{\text{disj}} &
   \alpha_{\text{disj}} & \alpha_{\text{div}} & \alpha_{\text{chain}} & \alpha_{\text{disj}} &
   \alpha_{\text{disj}} & \alpha_{\text{chain}} & \alpha_{\text{div}} \\
 \alpha_{\text{chain}} & \alpha_{\text{div}} & \alpha_{\text{chain}} & \alpha_{\text{div}} & \alpha_{\text{recip}}
   & 1 & \alpha_{\text{chain}} & \alpha_{\text{div}} & \alpha_{\text{disj}} & \alpha_{\text{disj}} &
   \alpha_{\text{chain}} & \alpha_{\text{conv}} & \alpha_{\text{disj}} & \alpha_{\text{disj}} &
   \alpha_{\text{chain}} & \alpha_{\text{conv}} & \alpha_{\text{disj}} & \alpha_{\text{disj}} &
   \alpha_{\text{conv}} & \alpha_{\text{chain}} \\
 \alpha_{\text{conv}} & \alpha_{\text{chain}} & \alpha_{\text{conv}} & \alpha_{\text{chain}} &
   \alpha_{\text{conv}} & \alpha_{\text{chain}} & 1 & \alpha_{\text{recip}} & \alpha_{\text{disj}} &
   \alpha_{\text{disj}} & \alpha_{\text{disj}} & \alpha_{\text{disj}} & \alpha_{\text{div}} &
   \alpha_{\text{chain}} & \alpha_{\text{disj}} & \alpha_{\text{disj}} & \alpha_{\text{div}} &
   \alpha_{\text{chain}} & \alpha_{\text{div}} & \alpha_{\text{chain}} \\
 \alpha_{\text{chain}} & \alpha_{\text{div}} & \alpha_{\text{chain}} & \alpha_{\text{div}} & \alpha_{\text{chain}}
   & \alpha_{\text{div}} & \alpha_{\text{recip}} & 1 & \alpha_{\text{disj}} & \alpha_{\text{disj}} &
   \alpha_{\text{disj}} & \alpha_{\text{disj}} & \alpha_{\text{chain}} & \alpha_{\text{conv}} &
   \alpha_{\text{disj}} & \alpha_{\text{disj}} & \alpha_{\text{chain}} & \alpha_{\text{conv}} &
   \alpha_{\text{chain}} & \alpha_{\text{conv}} \\
 \alpha_{\text{chain}} & \alpha_{\text{conv}} & \alpha_{\text{div}} & \alpha_{\text{chain}} & \alpha_{\text{disj}}
   & \alpha_{\text{disj}} & \alpha_{\text{disj}} & \alpha_{\text{disj}} & 1 & \alpha_{\text{recip}} &
   \alpha_{\text{conv}} & \alpha_{\text{chain}} & \alpha_{\text{conv}} & \alpha_{\text{chain}} &
   \alpha_{\text{chain}} & \alpha_{\text{div}} & \alpha_{\text{chain}} & \alpha_{\text{div}} &
   \alpha_{\text{disj}} & \alpha_{\text{disj}} \\
 \alpha_{\text{div}} & \alpha_{\text{chain}} & \alpha_{\text{chain}} & \alpha_{\text{conv}} & \alpha_{\text{disj}}
   & \alpha_{\text{disj}} & \alpha_{\text{disj}} & \alpha_{\text{disj}} & \alpha_{\text{recip}} & 1 &
   \alpha_{\text{chain}} & \alpha_{\text{div}} & \alpha_{\text{chain}} & \alpha_{\text{div}} &
   \alpha_{\text{conv}} & \alpha_{\text{chain}} & \alpha_{\text{conv}} & \alpha_{\text{chain}} &
   \alpha_{\text{disj}} & \alpha_{\text{disj}} \\
 \alpha_{\text{chain}} & \alpha_{\text{conv}} & \alpha_{\text{disj}} & \alpha_{\text{disj}} & \alpha_{\text{div}}
   & \alpha_{\text{chain}} & \alpha_{\text{disj}} & \alpha_{\text{disj}} & \alpha_{\text{conv}} &
   \alpha_{\text{chain}} & 1 & \alpha_{\text{recip}} & \alpha_{\text{conv}} & \alpha_{\text{chain}} &
   \alpha_{\text{div}} & \alpha_{\text{chain}} & \alpha_{\text{disj}} & \alpha_{\text{disj}} &
   \alpha_{\text{chain}} & \alpha_{\text{div}} \\
 \alpha_{\text{div}} & \alpha_{\text{chain}} & \alpha_{\text{disj}} & \alpha_{\text{disj}} & \alpha_{\text{chain}}
   & \alpha_{\text{conv}} & \alpha_{\text{disj}} & \alpha_{\text{disj}} & \alpha_{\text{chain}} &
   \alpha_{\text{div}} & \alpha_{\text{recip}} & 1 & \alpha_{\text{chain}} & \alpha_{\text{div}} &
   \alpha_{\text{chain}} & \alpha_{\text{conv}} & \alpha_{\text{disj}} & \alpha_{\text{disj}} &
   \alpha_{\text{conv}} & \alpha_{\text{chain}} \\
 \alpha_{\text{chain}} & \alpha_{\text{conv}} & \alpha_{\text{disj}} & \alpha_{\text{disj}} & \alpha_{\text{disj}}
   & \alpha_{\text{disj}} & \alpha_{\text{div}} & \alpha_{\text{chain}} & \alpha_{\text{conv}} &
   \alpha_{\text{chain}} & \alpha_{\text{conv}} & \alpha_{\text{chain}} & 1 & \alpha_{\text{recip}} &
   \alpha_{\text{disj}} & \alpha_{\text{disj}} & \alpha_{\text{div}} & \alpha_{\text{chain}} & \alpha_{\text{div}}
   & \alpha_{\text{chain}} \\
 \alpha_{\text{div}} & \alpha_{\text{chain}} & \alpha_{\text{disj}} & \alpha_{\text{disj}} & \alpha_{\text{disj}}
   & \alpha_{\text{disj}} & \alpha_{\text{chain}} & \alpha_{\text{conv}} & \alpha_{\text{chain}} &
   \alpha_{\text{div}} & \alpha_{\text{chain}} & \alpha_{\text{div}} & \alpha_{\text{recip}} & 1 &
   \alpha_{\text{disj}} & \alpha_{\text{disj}} & \alpha_{\text{chain}} & \alpha_{\text{conv}} &
   \alpha_{\text{chain}} & \alpha_{\text{conv}} \\
 \alpha_{\text{disj}} & \alpha_{\text{disj}} & \alpha_{\text{chain}} & \alpha_{\text{conv}} & \alpha_{\text{div}}
   & \alpha_{\text{chain}} & \alpha_{\text{disj}} & \alpha_{\text{disj}} & \alpha_{\text{chain}} &
   \alpha_{\text{conv}} & \alpha_{\text{div}} & \alpha_{\text{chain}} & \alpha_{\text{disj}} &
   \alpha_{\text{disj}} & 1 & \alpha_{\text{recip}} & \alpha_{\text{conv}} & \alpha_{\text{chain}} &
   \alpha_{\text{chain}} & \alpha_{\text{div}} \\
 \alpha_{\text{disj}} & \alpha_{\text{disj}} & \alpha_{\text{div}} & \alpha_{\text{chain}} & \alpha_{\text{chain}}
   & \alpha_{\text{conv}} & \alpha_{\text{disj}} & \alpha_{\text{disj}} & \alpha_{\text{div}} &
   \alpha_{\text{chain}} & \alpha_{\text{chain}} & \alpha_{\text{conv}} & \alpha_{\text{disj}} &
   \alpha_{\text{disj}} & \alpha_{\text{recip}} & 1 & \alpha_{\text{chain}} & \alpha_{\text{div}} &
   \alpha_{\text{conv}} & \alpha_{\text{chain}} \\
 \alpha_{\text{disj}} & \alpha_{\text{disj}} & \alpha_{\text{chain}} & \alpha_{\text{conv}} & \alpha_{\text{disj}}
   & \alpha_{\text{disj}} & \alpha_{\text{div}} & \alpha_{\text{chain}} & \alpha_{\text{chain}} &
   \alpha_{\text{conv}} & \alpha_{\text{disj}} & \alpha_{\text{disj}} & \alpha_{\text{div}} &
   \alpha_{\text{chain}} & \alpha_{\text{conv}} & \alpha_{\text{chain}} & 1 & \alpha_{\text{recip}} &
   \alpha_{\text{div}} & \alpha_{\text{chain}} \\
 \alpha_{\text{disj}} & \alpha_{\text{disj}} & \alpha_{\text{div}} & \alpha_{\text{chain}} & \alpha_{\text{disj}}
   & \alpha_{\text{disj}} & \alpha_{\text{chain}} & \alpha_{\text{conv}} & \alpha_{\text{div}} &
   \alpha_{\text{chain}} & \alpha_{\text{disj}} & \alpha_{\text{disj}} & \alpha_{\text{chain}} &
   \alpha_{\text{conv}} & \alpha_{\text{chain}} & \alpha_{\text{div}} & \alpha_{\text{recip}} & 1 &
   \alpha_{\text{chain}} & \alpha_{\text{conv}} \\
 \alpha_{\text{disj}} & \alpha_{\text{disj}} & \alpha_{\text{disj}} & \alpha_{\text{disj}} & \alpha_{\text{chain}}
   & \alpha_{\text{conv}} & \alpha_{\text{div}} & \alpha_{\text{chain}} & \alpha_{\text{disj}} &
   \alpha_{\text{disj}} & \alpha_{\text{chain}} & \alpha_{\text{conv}} & \alpha_{\text{div}} &
   \alpha_{\text{chain}} & \alpha_{\text{chain}} & \alpha_{\text{conv}} & \alpha_{\text{div}} &
   \alpha_{\text{chain}} & 1 & \alpha_{\text{recip}} \\
 \alpha_{\text{disj}} & \alpha_{\text{disj}} & \alpha_{\text{disj}} & \alpha_{\text{disj}} & \alpha_{\text{div}} &
   \alpha_{\text{chain}} & \alpha_{\text{chain}} & \alpha_{\text{conv}} & \alpha_{\text{disj}} &
   \alpha_{\text{disj}} & \alpha_{\text{div}} & \alpha_{\text{chain}} & \alpha_{\text{chain}} &
   \alpha_{\text{conv}} & \alpha_{\text{div}} & \alpha_{\text{chain}} & \alpha_{\text{chain}} &
   \alpha_{\text{conv}} & \alpha_{\text{recip}} & 1 \\
\end{array}
\right)}
\end{align*}
The problem of generating a Gaussian random vector with a given covariance matrix is an exercise in spectral theory. In particular one needs to compute the positive semi-definite square root of the matrix, $\bs^\frac12$, and act on a Gaussian iid random vector $\bo \in \RRR^E$ with this matrix. We note the following remarkable algebraic facts about  the spectrum of the covariance $\bs(\ba)$, which do not appear to have been noted previously:
\begin{itemize}
\item For fixed values of $\ba$ the covariance matrix $\bs(\ba)$ has
  at most five distinct eigenvalues, independent of the size of the matrix, so the multiplicities of the eigenvalues are typically very large.
\item The eigenvalues of the covariance matrix can be computed explicitly as a function of $\ba$ for all $N$. 
\item The dependence of the eigenvalues on the components of $\ba$ is simple: the eigenvalues are either linear functions in the components of $\ba$, or are algebraic of degree two (they are the roots of a quadratic). 
\end{itemize}
In particular for any number of vertices $N$ there are only five distinct eigenvalues of $\bs(\ba)$, which are given by 
\begin{align*}
&\lambda_1 = 1 +\alpha_{\text{recip}} +(N-2) \alpha_{\text{conv}}+ 2 (N-2) \alpha_{\text{chain}}+(N-2) \alpha_{\text{div}}  +(N-2)(N-3)
   \alpha_{\text{disjoint}}\\
&\lambda_2 = 1 -\alpha_{\text{recip}} -\alpha_{\text{conv}} +2 \alpha_{\text{chain}}  -\alpha_{\text{div}}  
  \\
&\lambda_3 = 1+\alpha_{\text{recip}}-\alpha_{\text{conv}} -2 \alpha_{\text{chain}} 
  -\alpha_{\text{div}}+2  \alpha_{\text{disjoint}}  \\
&\lambda_{4/5}  = 1+\frac{(N-3)}{2} \alpha_{\text{conv}} - \alpha_{\text{chain}} +\frac{(N-3)}{2} \alpha_{\text{div}} - (N-3)
   \alpha_{\text{disjoint}}\pm \frac{\tau}{2} 
\end{align*}
where 
\begin{align*}
\tau^2 = N(N-2) (\alpha_{\text{conv}}-\alpha_{\text{div}})^2 + (\alpha_{\text{conv}} + \alpha_{\text{div}} - 2 (N-3)
(\alpha_{\text{chain}}-\alpha_{\text{disj}}) - 2 \alpha_{\text{recip}})^2  
\end{align*}
In order to be admissible as a covariance matrix $\bs(\ba)$ must be
positive definite. One consequence of the above formulae is that  one
can compute, reasonably explicitly, the region of parameter space in
which $\bs(\ba)$ is positive definite: it is given by the intersection of
three half-planes ($\lambda_{1,2,3}>0$) and the region cut out by two
hyperboloids ($\lambda_{4,5}>0$). We show that the algebra satisfied
by the covariance matrices is isomorphic to an algebra of matrices of
fixed and rather modest size: $7\times 7$ in the case of the
Nykamp-Zhao algorithm. Any linear algebraic operation that one might
want to perform on an $N(N-1) \times N(N-1)$ covariance matrix can
instead be done on a representative from the algebra of $7\times7$
matrices, and the results can be translated back to the $N(N-1) \times N(N-1)$ covariance matrix. 
This includes computation of eigenvalues, inversion and extraction of
the square root, the main operation
required to  apply the SONETS algorithm. To reiterate:  the existence of a map to a
seven dimensional algebra implies that all of these linear algebraic
operations can be done in time {\em independent of the size of the matrix}.  

The basic algebraic explanation for this collection of facts is that there is, underlying the SONETS method, an algebraic object known as an association scheme or a coherent configuration\cite{higman.1975,cameron2003coherent,Zieschang,Bannai.Ito}. Association schemes and coherent configurations arise in a number of applications, including coding theory\cite{Delsarte} and the design of experiments\cite{Bailey.Book}. These structures represent a generalization of the notion of a group encoding certain nice pairwise relations between elements of a set.  The SONETS method arises naturally from a homogeneous coherent configuration (HCC), a  non-commuting generalization of the classical Johnson $J(N,2)$ association scheme. The relations in the HCC roughly correspond to motifs in the network. Actually, as will be explained, the relations represent a slight refinement of the idea of a motif. One of the basic facts about association schemes and coherent configurations is that they give rise to a natural finite dimensional algebra, the Bose-Meisner algebra.  The algebra underlying the SONETS algorithm as described by Nykamp and Zhao has degree seven, implying that the number of distinct eigenvalues  of the covariance matrix is at most seven independent of the size of the network. For SONETS the number is actually less, and there are only five distinct eigenvalues. 

We refer the interested reader to any number of texts on coherent configurations and association schemes\cite{Bailey.Book,Zieschang,Bannai.Ito}, but in the interests of making this paper relatively self-contained we will include proofs of the results that we need.


\subsection{Association Schemes and Coherent Configurations}

Association schemes and coherent configurations are algebraic structures that arise in a number of areas including statistics, particularly the theory of experimental design, and coding theory. They are also used as a tool in abstract algebra for studying permutation groups, and can be thought of as representing a generalization of group theory and the associated representation theory. 

A coherent configuration can be described as follows.  
\begin{definition} \label{def:configuration}
Given an index set $\X$ and the product set $\X \times \X$ of ordered pairs of indices, a  coherent configuration is a set of $d+1$ relations $\RR^{(0)}, \RR^{(1)}, \ldots,\RR^{(d)}$ given by subsets of $\X \times \X$ satisfying the following properties:
\begin{enumerate}
\item For every $x$ and $y$ in $\X$ there is a unique relation $\RR^{(k)}$ such that the ordered pair $(x,y) \in \RR^{(k)}.$ In other words the relations partition $\X \times \X$. 
\item There is a subset of relations $\RR^{(0)}, \RR^{(1)}, \ldots, \RR^{(j)}$ that partition the diagonal set \\$\{ (x,x) ~~\rvert ~~ x \in \X \}$.
\item If ${\mathcal R}^{(k)}$ is a relation in the set the adjoint relation defined by \\$(\RR^{(k)})^\dagger := \{ (y,x)~~ \rvert ~~(x,y) \in \RR^{(k)} \}$ is also a relation in the set.  
\item Given any pair $(x,y) \in \RR^{(i)}$ the number of elements $z$ such that $(x,z) \in \RR^{(k)}$ and $(z,y) \in \RR^{(j)}$ depends only on $i,j,k$ and not on the individual $(x,y)$.  
\end{enumerate}
A homogeneous coherent configuration (HCC) is a coherent configuration for which property (2) above is replaced by the stronger condition that one of the relations be the identity relation.
\begin{itemize}
\item[(2')] $\RR^{(0)} =\{ (x,x) | x \in \X \}$. 
\end{itemize}
An association scheme is is a homogeneous coherent configuration for
which property (3) above is replaced by the stronger condition that the relations be symmetric.
\begin{itemize}
\item[(3')] $(\RR^{(k)})^\dagger = \RR^{(k)}$. 
\end{itemize}
\end{definition}
We should warn the reader that we are adhering to the terminology of Cameron\cite{cameron2003coherent} and the earlier work of Higman\cite{higman.1975} but that this is not followed by all authors. For instance in the work of Hanaki and Miyamoto\cite{Hanaki.I,Hanaki.II} what they call association schemes would be called homogeneous coherent configurations in the nomenclature above: there is a single identity relation but relations are not necessarily symmetric. 

For any coherent configuration there exists a set of non-negative integers
$\rho^{(k)}_{ij}$, commonly referred to as the structure constants or intersection
numbers. These integers play a fundamental role in the algebraic
theory of coherent configurations, and are defined as follows.

\begin{definition} \label{def:constants}
The integers $\rho^{(k)}_{ij}$ are defined as follows: given a representative pair $(x,y)\in \RR^{(i)}$ 
\begin{align*}
\rho^{(k)}_{ij} = \left\vert \{ z \in \X ~~\rvert ~~(x,z) \in \RR^{(k)} \text{ and } (z,y) \in \RR^{(j)} \} \right\vert.  
\end{align*}
In other words $\rho^{(k)}_{ij}$ counts the number of $z$ such that
$(x,z)$ belongs to relation $k$ and $(z,y)$ belongs to relation $j$. 
Note that this well defined by $(4)$ of Definition
\ref{def:configuration}: this only depends on the relation to which
$(x,y)$ belongs, and not on the individual $x$ and $y$. 
\end{definition}

For each relation we define a corresponding adjacency matrix, as follows: $\R^{(k)} = \chi_{\RR^{(k)}}$, where $\chi$ is the usual characteristic function. In other words,
\begin{align*}
\R_{ij}^{(k)} =
\begin{cases}
1 & \mbox{if $(i,j) \in \RR^{(k)}$,}
\\
0 & \mbox{if $(i,j) \notin \RR^{(k)}$.}
\end{cases}
\end{align*}
This defines a collection of $(d+1)$ $(0,1)$ matrices representing the relations. We wish to note three important facts about the matrices $\{ \R^{(i)}\}_{i=0}^d$, which all follow immediately from the facts above:
\begin{itemize}
\item The matrices $\{ \R^{(i)}\}_{i=0}^d$ are linearly independent and orthogonal under the usual matrix inner product $\langle \A, \B \rangle = \tr(\A^\top \B)$.
\item The matrices form a closed algebra, in the sense that $\R^{(k)} \R^{(j)} = \sum_i \rho^{(k)}_{ij} \R^{(i)} $. Here the product is the usual matrix product. 
\item The matrices satisfy $\sum_i \R^{(i)} = \1_{|\X| \times |\X|}$, where $\1_{|\X| \times |\X|}$ is the $|\X| \times |\X|$ matrix with all entries equal to $1$. 
\end{itemize}
The structure constants $\rho^{(k)}_{ij}$ play an important role in
the theory as they  can be used to define an isomorphism of algebras that explains the special properties of covariance matrices in the SONETS scheme. We will address this in the following section. 

Note that, if the relations are all symmetric, then clearly $\rho^{(k)}_{ij} = \rho^{(j)}_{ik}$ and thus the corresponding adjacency matrices commute. While this will be true in the simplest example that we consider, that of the classical Johnson $J(N,2)$ scheme, it will not be true for most of the examples that we consider. 

One classical example of a coherent configuration is distance on a distance regular graph: a pair
of edges $(x,y)$ belongs to relation $\RR^{(k)}$ if $x$ distance $k$
from $y$ in the graph.  In this scheme all relations are obviously
symmetric, as the distance is symmetric, so this is actually an
association scheme. Throughout this paper we frequently reference the
Johnson association scheme: this arises from the distance regular
Johnson $(N,2)$ graphs in exactly this way. 

A second example of a coherent configuration is given by any group. There is one relation for each group element $g$, and a pair of elements $e,f$ are in $\RR^{g}$ if $e = g f.$ For this scheme the adjacency matrices are permutation matrices -- there is one non-zero entry in each row and column -- and the adjoint relation is $\RR^{g^{-1}}$. This scheme is, of course, typically not symmetric. In fact it is not hard to see that any association scheme where the adjacency matrices are permutation matrices comes from a group in this way, so in this way groups are a special case of association schemes where the scheme has the largest possible number of relations, and the adjacency matrix for any relation is a permutation matrix. 

In the context of networks the relations roughly correspond to two edge motifs in the network. The correspondence is not quite exact, for reasons that we will see  shortly. In fact the idea of a relation is slightly  more precise than that of a motif, and each relation can be thought of as a motif together with a role for each edge in the motif.


\begin{example}[The Johnson Scheme]
Consider the complete undirected graph with $E = \frac{N(N-1)}{2}$ edges. There is a natural association scheme on the edge set with three relations:
\begin{itemize}
\item $\RR^{(0)} = \{ (x,y) ~~\vert ~~x = y\}$ -- the identity relation. The corresponding matrix is the identity matrix
\item $\RR^{(1)} = \{ (x,y) ~~\vert ~~x \text{ and } y \text{ share one vertex}\}$
\item $\RR^{(2)} = \{ (x,y) ~~\vert ~~x \text{ and } y \text{ share no vertices}\}$.
\end{itemize}
One way to think about this is as follows: edges in the complete
undirected graph can be indexed by unordered pairs $(i,j)$, where the
edge connects vertices $i$ and $j$. Two edges $x$ and $y$ are in
$\RR^{(k)}$ if the number of elements in the intersection satisfies
$|x \cap y| = 2-k.$ This is distance in the
line graph of the complete graph -- the Johnson graph -- a distance
regular graph.  In this case
one can imagine using the SONETS idea to generate random undirected
graphs with correlations among the edges. We will do this later in the
paper. There are three motifs here,
corresponding to the three relations: the identity motif (the edges
are the same), the adjacent motif (the edges share a vertex) and the
disjoint motif (the edges do not share a vertex). For $N=4$ the
covariance matrix would be given by  
\begin{align*}
\R^{(0)} + \alpha_1 \R^{(1)} + \alpha_2 \R^{(2)} = \left(
\begin{array}{cccccc}
 1 & {\alpha_1} & {\alpha_1} & {\alpha_1} & {\alpha_1} & {\alpha_2} \\
 {\alpha_1} & 1 & {\alpha_1} & {\alpha_1} & {\alpha_2} & {\alpha_1} \\
 {\alpha_1} & {\alpha_1} & 1 & {\alpha_2} & {\alpha_1} & {\alpha_1} \\
 {\alpha_1} & {\alpha_1} & {\alpha_2} & 1 & {\alpha_1} & {\alpha_1} \\
 {\alpha_1} & {\alpha_2} & {\alpha_1} & {\alpha_1} & 1 & {\alpha_1} \\
 {\alpha_2} & {\alpha_1} & {\alpha_1} & {\alpha_1} & {\alpha_1} & 1 \\
\end{array}
\right)
\end{align*}
where $\alpha_1$ and $\alpha_2$ are constants representing the correlations between adjacent edges and disjoint edges. It is easy to confirm that the covariance is a linear combination of three commuting matrices and therefore that the eigenvalues are linear in $\alpha_1$ and $\alpha_2$. In fact the eigenvalues are $\lambda = 1 + 4 \alpha_1+\alpha_2$, with multiplicity one, $\lambda = 1 - 2 \alpha_1+\alpha_2$, with multiplicity two, and $\lambda = 1 -\alpha_2$, with multiplicity three.
\end{example}

The next example is the most important one for the purposes of this
paper, and underlies the SONETS algorithm detailed by Nykamp and
Zhao.


\begin{example}[The Nykamp-Zhao Homogeneous Coherent Configuration]
This HCC is defined on a complete directed graph
with $E=N(N-1)$ edges, and has seven relations on ordered pairs of edges
$(x,y)$.  
\begin{itemize}
\item{\bf  Identity}: Edges $x$ and $y$ are the same edge.
\item {\bf Reciprocal}: Edges $x$ and $y$ point between the same vertices but in opposite directions. 
\item {\bf Convergence}: Edges $x$ and $y$ point inward towards the same vertex. 
\item {\bf Chain}: Edge $x$ points inward to a vertex, edge $y$ points outward from the same vertex.
\item {\bf Anti-Chain}: Edge $x$ points outward from a vertex, edge $y$ points inward to the same vertex.
\item {\bf Divergence}: Edges $x$ and $y$ point outward from the same vertex. 
\item {\bf Disjoint}: Edges $x$ and $y$ do not share any vertices. 
\end{itemize}
Note that the Nykamp and Zhao discuss only four motifs: Reciprocal,
Convergence, Chain, and Divergence. The identity motif is, of course,
implicitly present in their work but is not discussed. The chain motif
translates into two relations, Chain and Anti-Chain. Both are required
in order for the relations to form a coherent configuration although
one can of course choose the coefficients of the two to be equal,
recovering the motif. Nykamp and Zhao do not discuss the disjoint motif, though in principle it could be considered in the same framework.
\end{example}


\begin{remark}
Some things to note about this coherent configuration. There are two
relations (Identity and Reciprocal) where the edges share two
vertices. There are four relations (Convergence, Chain, Anti-chain and
Divergence) where the edges share a single vertex. There is a single
relation (Disjoint) where the edges do not share any vertices. These
generalize the three relations in the Johnson scheme: distance zero,
one and two respectively. Five of the relations (Identity, Reciprocal,
Convergence, Divergence and Disjoint) are symmetric, while the
remaining two (Chain and Anti-Chain) are adjoints of one another. The
symmetric relations exactly correspond to the associated motifs, while
non-symmetric motifs generate a pair of relations. For instance the first row of $\R^{\text{Chain}}$ gives the edges that are outgoing from the vertex that edge one is incoming to. The first row of $\R^{\text{Anti-Chain}}$, however, gives the edges that are incoming to the vertex that edge one is outgoing from. Of course the union $\RR^{\text{Chain}} \cup \RR^{\text{Anti-Chain}}$ is symmetric and corresponds to the motif exactly. Six of the seven relations are illustrated in Figure (\ref{fig:DJohnson}).  
\end{remark}

\begin{figure}[h]
\centering
\begin{tabular}{|c|c|c|}
\hline
\subf{\includegraphics[width=40mm]{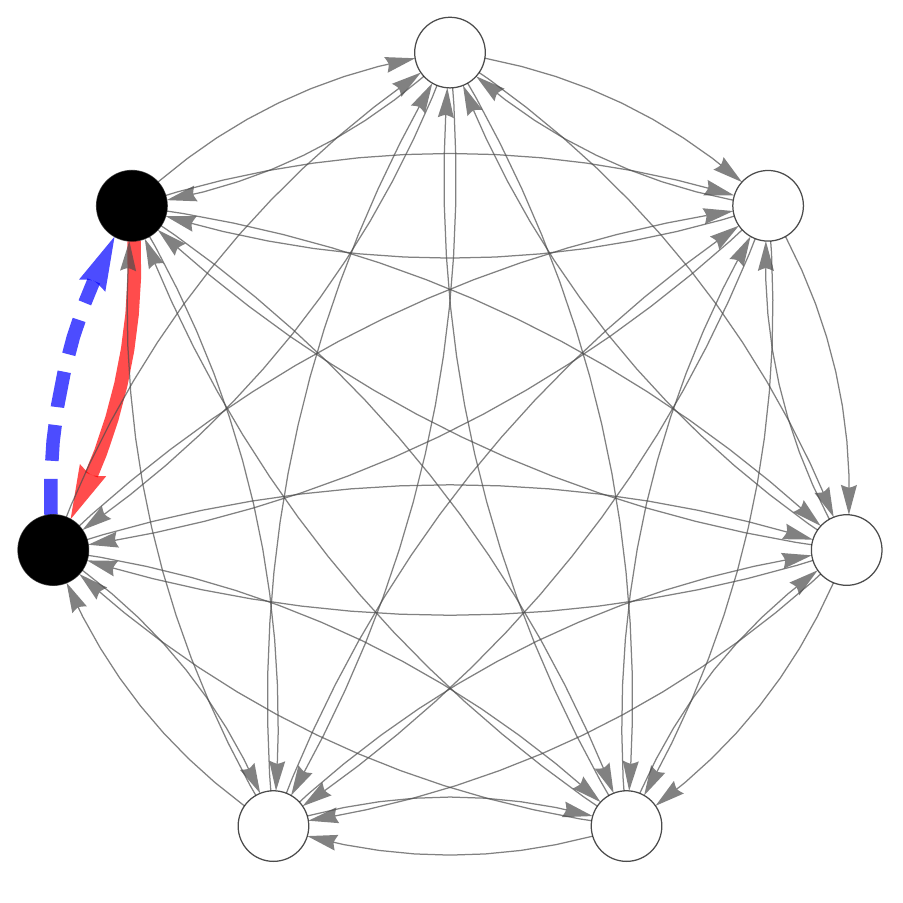}}{Reciprocal Motif}
&
\subf{\includegraphics[width=40mm]{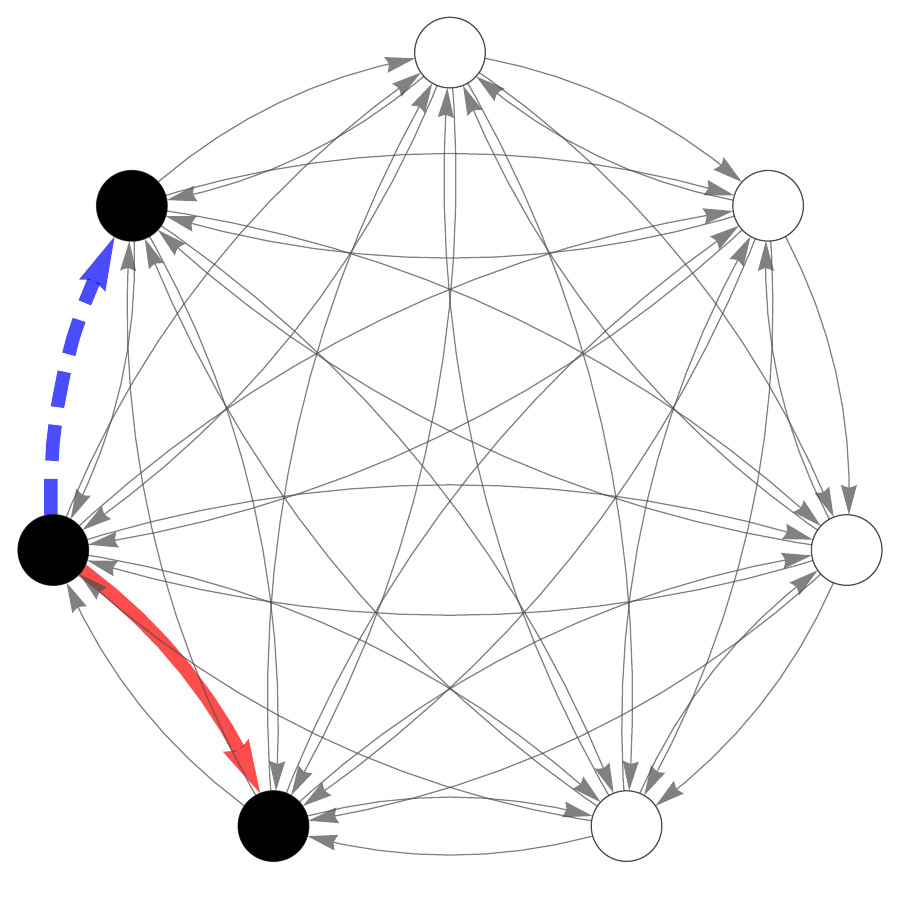}}{Divergent Motif} 
&
\subf{\includegraphics[width=40mm]{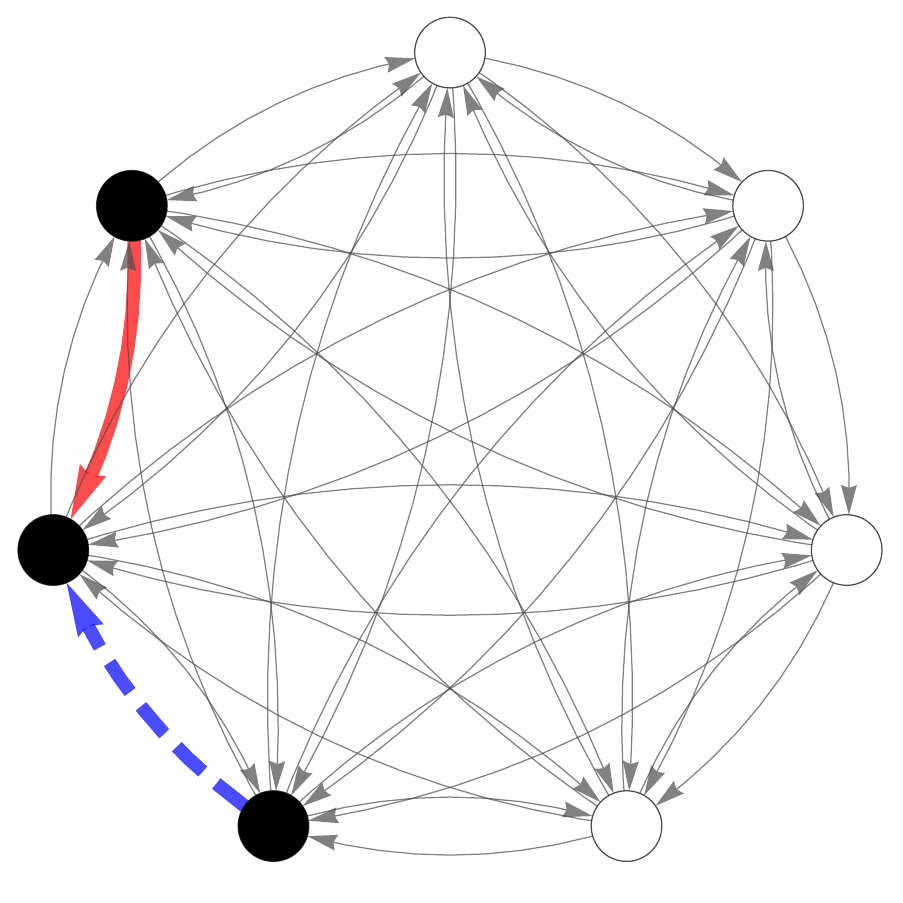}}{Convergent Motif}
\\
\hline
\subf{\includegraphics[width=40mm]{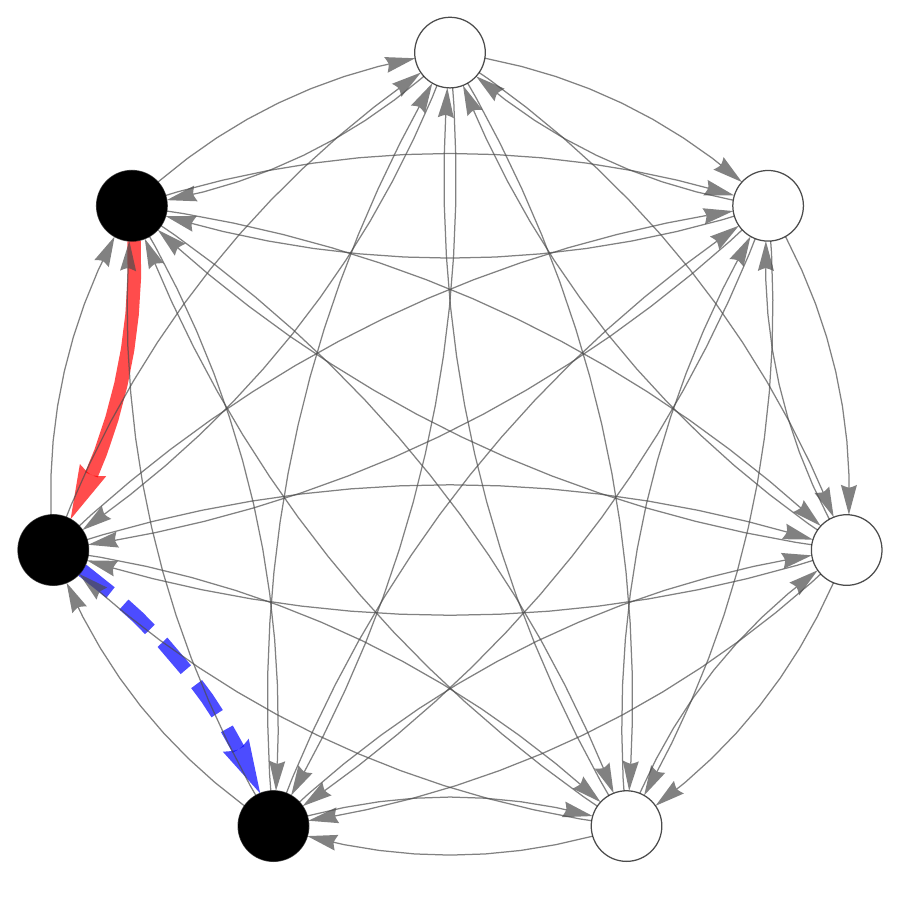}}{Chain Motif}
&
\subf{\includegraphics[width=40mm]{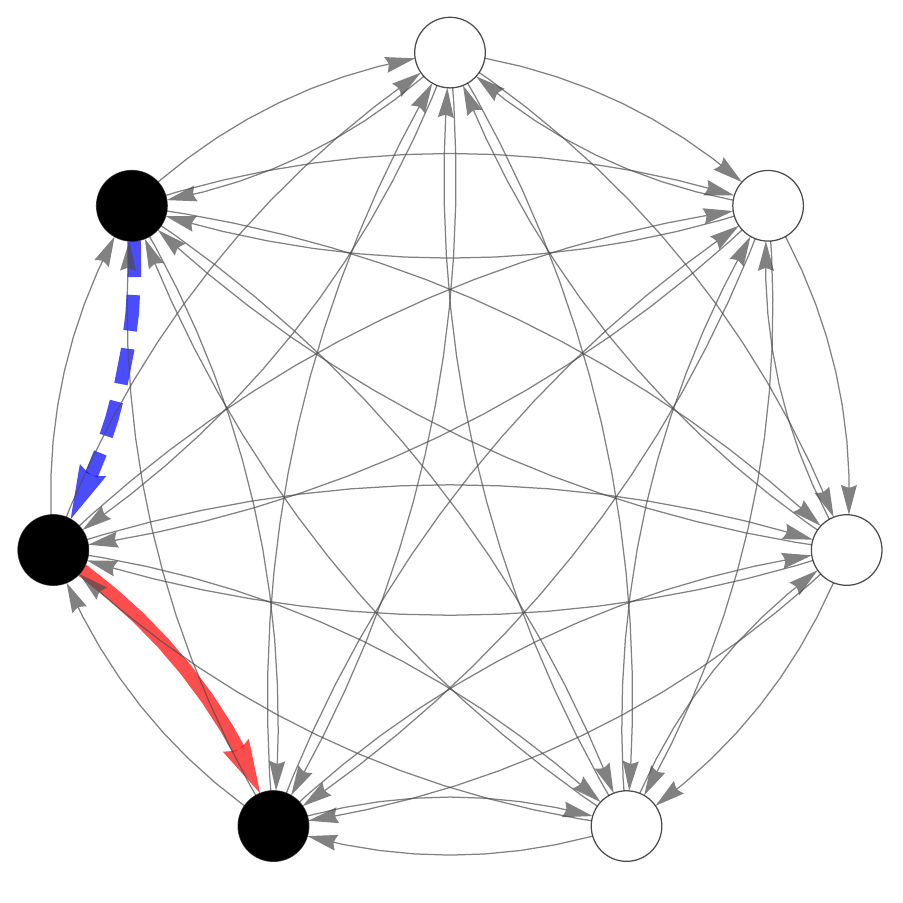}}{Anti-Chain Motif}
&
\subf{\includegraphics[width=40mm]{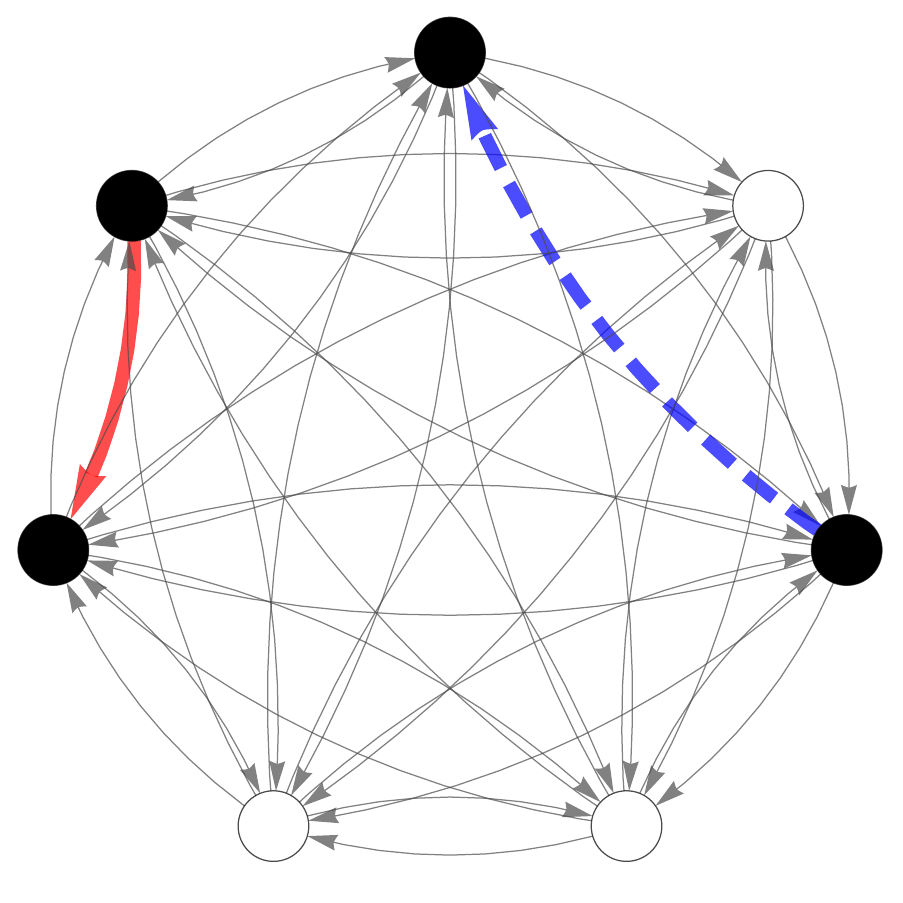}}{Disjoint Motif}
\\
\hline
\end{tabular}
\caption{Six of the relations in the homogeneous coherent configuration described in the text. We have not depicted the identity motif. The thick red edge denotes the first edge in the configuration, the thick dashed blue edge denotes the second edge. }
\label{fig:DJohnson} 
\end{figure}
We next illustrate how to compute the structure constants $\rho^{(k)}_{ij}.$ These integers count how many edges are in a certain relation with two other edges. Specifically given that the pair of edges $(x,y)$ is in relation $i$ the coefficient $\rho^{(k)}_{ij}$ counts the number of edges $z$  such that $(x,z)$ is in relation $k$ and $(z,y)$ is in relation $j$. 
\begin{example}[The Johnson scheme multiplication laws]
The classical Johnson scheme $J(N,2)$ has 3 relations. The identity relation gives the identity matrix, so $\R^{(0)} = \I$. The remaining multiplication laws are 
\begin{align}
\R^{(1)} \R^{(1)} &= 2(N-2) \R^{(0)} + (N-2) \R^{(1)} + 4 \R^{(2)},
\\
\R^{(1)} \R^{(2)} &= \R^{(2)} \R^{(1)} = (N-3) \R^{(1)} + 2(N-4) \R^{(2)},
\\
\R^{(2)} \R^{(2)} &= \frac{(N-2)(N-3)}{2} \R^{(0)} +  \frac{(N-3)(N-4)}{2} \R^{(1)} + \frac{(N-4)(N-5)}{2} \R^{(2)}.
\end{align}
We illustrate the computation of one of these terms, say the coefficient of $\R^{(1)}$ in the expansion of $\R^{(1)} \R^{(1)}$. Obviously we have $(\R^{(1)} \R^{(1)})_{xy } = \sum_z \R^{(1)}_{x z} \R^{(1)}_{zy}$ To compute the coefficient of $\R^{(1)}$ in this product we choose edges $x$ and $y$ such that the pair $(x,y)$ are in the $\R^{(1)}$ relation, meaning they share exactly one vertex. By hypothesis it doesn't matter which pair we choose, so long as they are in the right relation, so choose $x = \{1,2\}$ and $y =\{1,3\}$. The coefficient of $\R^{(1)}$ is the number of edges $z$ that share exactly one  vertex with $x = \{1,2\}$ and exactly one vertex with $y = \{1,3\}$. There are exactly $N-2$ such edges: the edges $\{1,x\}$ where $x \geq 4$ (of which there are $N-3$) and the edge $(2,3)$, giving a coefficient of $N-3+1=N-2$. As a second example we take the coefficient of $\R^{(2)}$ in the product $\R^{(1)} \R^{(2)}$. In this case $x$ and $y$ have to belong to relation $\R^{(2)}$, so take $x = \{1,2\}$ and $y = \{3,4\}$. We want to count the number of $z$ that have exactly one element in common with $\{1,2\}$ and no elements in common with $\{3,4\}$. These are exactly $z = \{1,x\}$ with $x \geq 5$ and $z = \{2,x\}$ with $x \geq 5$. Thus the coefficient of $\R^{(2)}$ in the product $\R^{(1)}
\R^{(2)}$ is $2(N-4)$. 
\end{example}
\begin{example}[The Nykamp-Zhao multiplication laws.]
The computation of one multiplication law satisfied by the adjacency matrices in the Nykamp-Zhao coherent configuration is illustrated in Figure \ref{fig:MultLaw} for $N=7$. This illustrates $\R^{\text{conv}} \R^{\text{disj}} = (N-3) \R^{\text{anti}} + (N-3) \R^{\text{div}} + (N-4) \R^{\text{disj}}$. In this figure the solid edge (red online) denotes the first edge ($x$ in the notation above), the dashed edge the last edge ($y$ above) and the dotted edges ($z$ above) are the intermediate edges to be counted. To compute the coefficient of $\R^{\text{anti}} $ in the product $\R^{\text{conv}} \R^{\text{disj}}$ we choose $(x,y)$ a pair of edges in the anti-chain configuration, and we count 
the number of edges $z$ that form the convergent motif with the first  edge and the disjoint motif with the second  edge. There are $7-3=4$ such edges. The other two subfigures show the computation of the structure coefficients of $\R^{\text{div}}$ and $\R^{\text{disj}}$ respectively. It is easy to see that the remaining structure constants must be zero. For instance the coefficient of $\R^{\text{recip}}$ would count the number of edges that are convergent with one edge and disjoint from the reciprocal edge. There are clearly no edges that satisfy both of those conditions, and thus the coefficient of $\R^{\text{recip}}$ in the expansion of $\R^{\text{conv}} \R^{\text{disj}}$ is zero.
\begin{figure}[h]
\centering
\begin{tabular}{|c|c|c|}
\hline
\subf{\includegraphics[width=40mm]{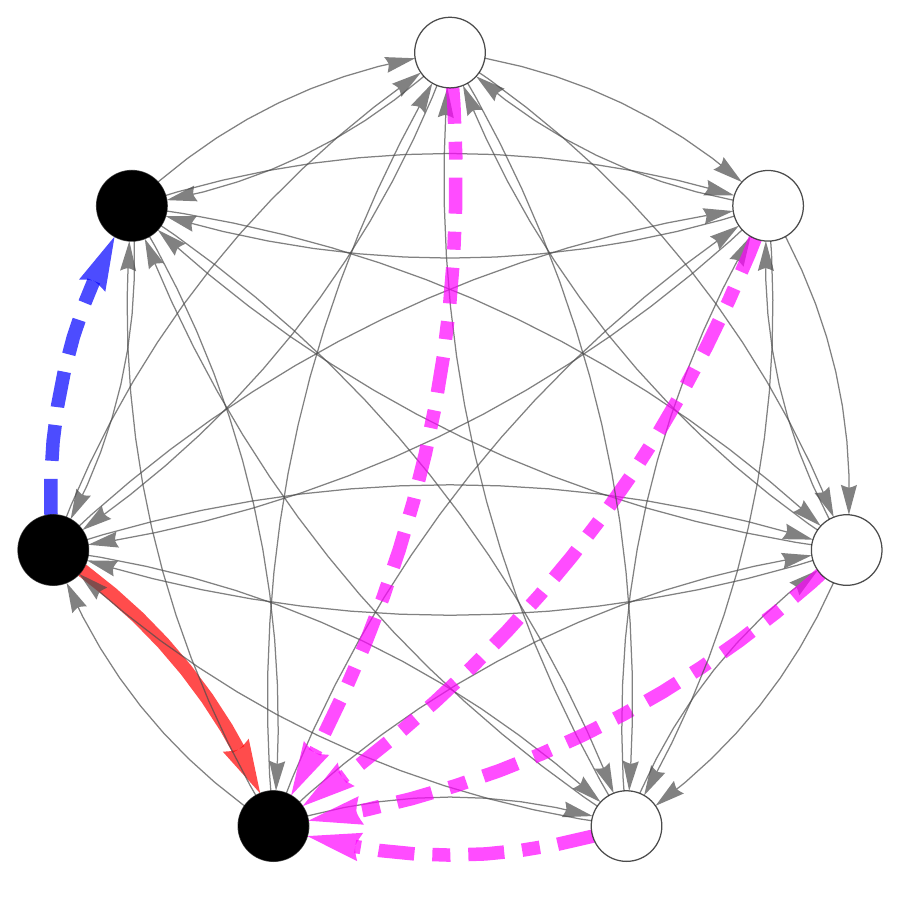}}{Coefficient of $\R^{\text{div}}$ }
&
\subf{\includegraphics[width=40mm]{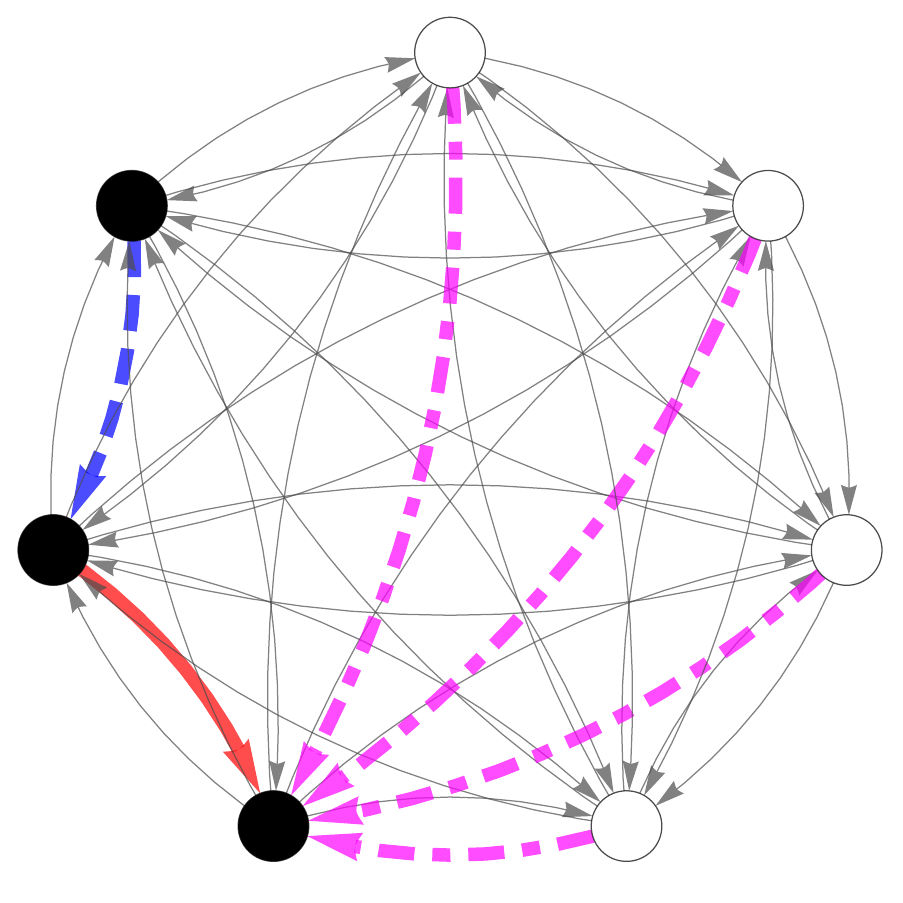}}{Coefficient of $\R^{\text{anti}}$}
& 
\subf{\includegraphics[width=40mm]{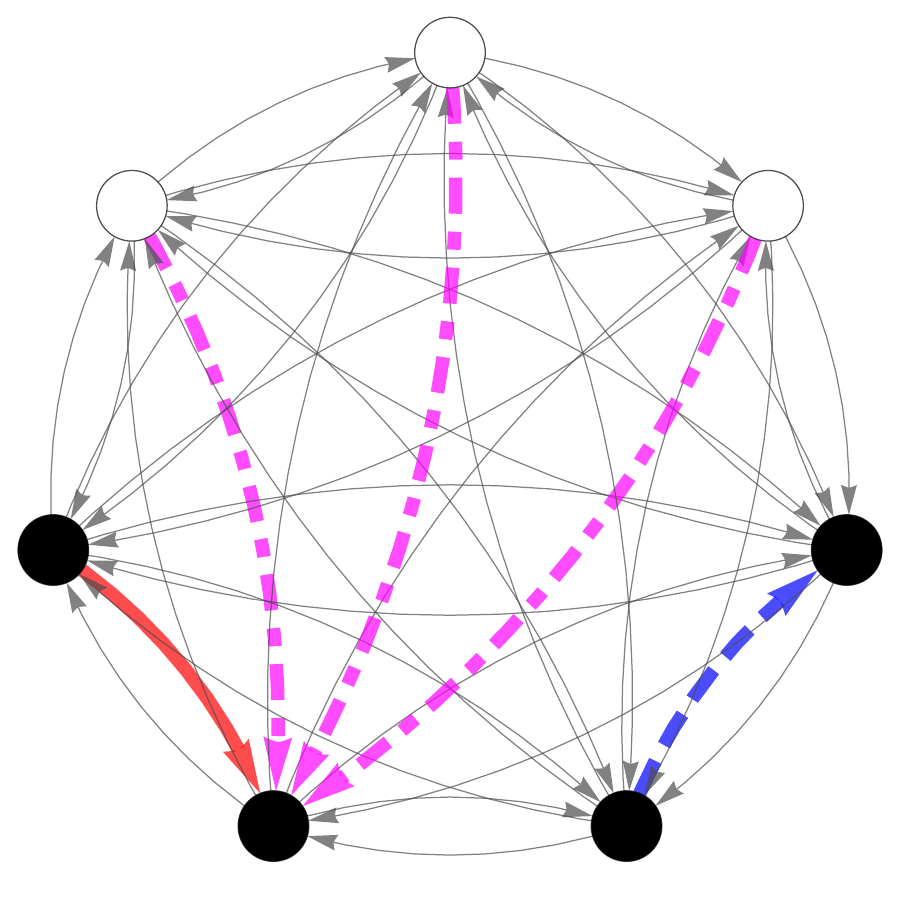}}{Coefficient of $\R^{\text{disj}}$} 
\\
\hline
\end{tabular}
\caption{The multiplication law $\R^{\text{conv}} \R^{\text{disj}} = (N-3) \R^{\text{div}} + (N-3) \R^{\text{anti}} + (N-4) \R^{\text{disj}}$.}
\label{fig:MultLaw} 
\end{figure}
\end{example}

Coherent configurations are often summarized by the table $\sum_k k
\R^{(k)}$, which indicates which subsets belong to which
relations. Since each relation consists of a $(0,1)$ matrix with
disjoint entries the relations follow immediately from such a table. 
Hanaki and Miyamoto give a complete list of all (up to isomorphism) homogeneous coherent
configurations (HCC's) with $|\X| \leq 34$ at {\em
  http://math.shinshu-u.ac.jp/~hanaki/as.} Note that Hanaki and
Miyamoto require an association scheme to have an identity relation
but do not require relations to be symmetric, thus what they call an
association scheme we would call a homogeneous coherent
configuration. There are 243 such HCC's of order 30.  If we take the
relations and edges to be ordered as follows: 
\begin{itemize}
\item $\RR^{(0)}$ -- Identity relation
\item $\RR^{(1)}$ -- Reciprocal relation
\item $\RR^{(2)}$ -- Divergent relation
\item $\RR^{(3)}$ -- Convergent relation
\item $\RR^{(4)}$ -- Chain relation
\item $\RR^{(5)}$ -- Anti-chain relation
\item $\RR^{(6)}$ -- Disjoint relation
\end{itemize}
\begin{align*}
\X = \{ &\Scale[0.75]{ (1, 2), (2, 1), (3, 2), (4, 2), (5, 2), (6, 2), (1, 3), (1, 4), (1,5), (1, 6), (3, 1), (4, 1), (5, 1), }\\
& \Scale[0.75]{(6, 1), (2, 3), (2, 4), (2, 5), (2, 6), (3, 4), (3, 5), (3, 6), (4,
  3), (5, 3), (6, 3), (4, 5), (4, 6),} \\
&\Scale[0.75]{(5, 4), (6, 4), (5, 6), (6, 5)\}} 
\end{align*}
then it is apparent that the Nykamp-Zhao HCC with $N=6$ is isomorphic to
number 99 in Hanaki and Miyamoto's list of schemes of order 30, while the Nykamp-Zhao HCC
with $N=5$ appears as number $51$ in their list of 95 schemes of order 20. 
\section{Algebraic Implications}

An important algebraic fact is that there is a homomorphism from
the  algebra of the adjacency matrices for the relations (of size
$N(N-1) \times N(N-1)$ in the case of the Nykamp-Zhao coherent
configuration) to an algebra of $(d+1) \times (d+1)$ matrices, where
$d+1$ is the number of relations in the coherent configuration
($7\times 7$ in this case). This is generally known as the intersection algebra. This is defined through


\begin{definition} \label{def:rho}
We define a linear map from the algebra generated by $\RR^{(0)}, \RR^{(1)}, \ldots,\RR^{(d)}$ to $(d+1) \times (d+1)$ matrices by
\begin{align*}
\rho \left(\sum_k \alpha_k \R^{(k)} \right)=\sum_k \alpha_k \rho(\R^{(k)})
\end{align*}
where $\rho(\R^{(k)})$ is defined to be the $(d+1)\times (d+1)$ matrix given by $\rho(\R^{(k)})_{ij} = \rho^{(k)}_{ij}$. 
\end{definition}

The algebraic coincidences in the SONETS algorithm basically arise from the fact that the map defined above is an injective homomorphism.  Essentially any linear-algebraic calculation that needs to be done on $\R^{(k)}$ can instead be done at the level of the $\rho^{(k)}$ instead, with the result then lifted back to the $\R^{(k)}$. To begin we state a proposition to this effect but defer the proof to the appendix.


\begin{proposition} \label{prop:rho}
Let $\rho$ be the linear map in Definition \ref{def:rho}, and let $\M$ denote a matrix in the algebra generated by $\RR^{(0)}, \RR^{(1)}, \ldots,\RR^{(d)}$. Then $\rho$ is an injective homomorphism satisfying the following properties:
\begin{itemize}
\item $\rho(\M)$ is diagonalizable if and only if $\M$ is.
\item $\lambda$ is an eigenvalue of $\rho(\M)$ if and only if it is for $\M$ (hence $\M$ has at most $d+1$ distinct eigenvalues).
\item The algebraic multiplicity of an eigenvalue $\lambda$ of $\M$ is
\begin{align*}
\mult(\lambda) = \rank ((\sum_k \beta_k \R^{(k)})_{\bb \in \mathcal{B}_\lambda})
\end{align*}
where $\mathcal{B}_\lambda$ is a basis for the generalized left eigenspace of $\rho(\M)$ for $\lambda$.
\end{itemize}
\end{proposition}

Note the special case that $\rho(\M)$ is diagonalizable if $\M$ is. This is significant to the SONETS problem because the $\R$ matrices, while large, are symmetric and thus diagonalizable. The corresponding elements of the intersection algebra, in contrast, are not symmetric or normal and there is no guarantee a priori that they are diagonalizable.

    
\begin{example}[The Johnson Scheme intersection algebra]
The first example is the classical Johnson scheme. Recall that $\R^{(0)} = \I$ and that
\begin{align*}
\R^{(1)} \R^{(1)} &= 2(N-2) \R^{(0)} + (N-2) \R^{(1)} + 4 \R^{(2)},
\\
\R^{(1)} \R^{(2)} &= \R^{(2)} \R^{(1)} = (N-3) \R^{(1)} + 2(N-4) \R^{(2)},
\\
\R^{(2)} \R^{(2)} &= \frac{(N-2)(N-3)}{2} \R^{(0)} +  \frac{(N-3)(N-4)}{2} \R^{(1)} + \frac{(N-4)(N-5)}{2} \R^{(2)}.
\end{align*}
This gives the homomorphism
\begin{align}
\rho(\R^{(0)}) &= \rho^{(0)} =
\begin{pmatrix}
1 & 0 & 0
\\
0 & 1 & 0
\\
0 & 0 & 1
\end{pmatrix},
\\
\rho(\R^{(1)}) &=  \rho^{(1)} =
\begin{pmatrix}
0 & 1 & 0
\\
2(N-2) & N-2 & 4
\\
0 & N-3 & 2(N-4)
\end{pmatrix},
\\
\rho(\R^{(2)}) &= \rho^{(2)} =
\begin{pmatrix}
0 & 0 & 1
\\
0 & N-3 & 2(N-4)
\\
\binom{N-2}{2} & \binom{N-3}{2} & \binom{N-4}{2}
\end{pmatrix}.
\end{align}
It is straightforward to verify that this is indeed a homomorphism, that the matrices $\rho^{(k)}$ satisfy the same algebraic identities as $\R^{(k)}$. The eigenvalues of a general linear combination $\I + \alpha_1 \rho^{(1)} + \alpha_2 \rho^{(2)}$ are easily computed to be 
\begin{align}
\lambda_0 &= 1 + 2 (N-2) \alpha_1 + \binom{N-2}{2} \alpha_2, \label{eqn:lambda0}
\\
\lambda_1 &= 1 + (N-4)  \alpha_1 - (N-3) \alpha_2, \label{eqn:lambda1}
\\
\lambda_2 &= 1-2\alpha_1 + \alpha_2. \label{eqn:lambda2}
\end{align}
These are the eigenvalues of  $\I + \alpha_1 \R^{(1)} + \alpha_2 \R^{(2)}$ as well, although the multiplicities obviously differ. It follows from the Perron-Frobenius theorem that $\lambda_0$ is a simple eigenvalue of $\I + \alpha_1 \R^{(1)} + \alpha_2 \R^{(2)}.$ It is possible, though tedious, to check that $\lambda_1$ has multiplicity $N-1$ and $\lambda_2$ has multiplicity $\frac{N(N-3)}{2}$, though we will not show this here. 
\end{example}
\begin{example}[The Nykamp-Zhao intersection algebra]
The following are the elements of the intersection algebra. We order the relations as follows: Identity, Reciprocal, Divergent, Chain, Anti-chain, Convergent, Disjoint. 
\begin{gather*}
\rho^{\text{id}} = \rho^{(0)} = \Scale[0.75]{\left(
\begin{array}{ccccccc}
 1 & 0 & 0 & 0 & 0 & 0 & 0 \\
 0 & 1 & 0 & 0 & 0 & 0 & 0 \\
 0 & 0 & 1 & 0 & 0 & 0 & 0 \\
 0 & 0 & 0 & 1 & 0 & 0 & 0 \\
 0 & 0 & 0 & 0 & 1 & 0 & 0 \\
 0 & 0 & 0 & 0 & 0 & 1 & 0 \\
 0 & 0 & 0 & 0 & 0 & 0 & 1 \\
\end{array}
\right)} \quad
\rho^{\text{recip}} = \rho^{(1)} = \Scale[0.75]{\left(
\begin{array}{ccccccc}
 0 & 1 & 0 & 0 & 0 & 0 & 0 \\
 1 & 0 & 0 & 0 & 0 & 0 & 0 \\
 0 & 0 & 0 & 1 & 0 & 0 & 0 \\
 0 & 0 & 1 & 0 & 0 & 0 & 0 \\
 0 & 0 & 0 & 0 & 0 & 1 & 0 \\
 0 & 0 & 0 & 0 & 1 & 0 & 0 \\
 0 & 0 & 0 & 0 & 0 & 0 & 1 \\
\end{array}
\right) }\\
\rho^{\text{div}} = \rho^{(2)} = \Scale[0.75]{\left(
\begin{array}{ccccccc}
 0 & 0 & 1 & 0 & 0 & 0 & 0 \\
 0 & 0 & 0 & 0 & 1 & 0 & 0 \\
 N-2 & 0 & N-3 & 0 & 0 & 0 & 0 \\
 0 & 0 & 0 & 0 & 0 & 1 & 1 \\
 0 & N-2 & 0 & 0 & N-3 & 0 & 0 \\
 0 & 0 & 0 & 1 & 0 & 0 & 1 \\
 0 & 0 & 0 & N-3 & 0 & N-3 & N-4 \\
\end{array}
\right) }\\
\rho^{\text{chain}} = \rho^{(3)} = \Scale[0.75]{\left(
\begin{array}{ccccccc}
 0 & 0 & 0 & 1 & 0 & 0 & 0 \\
 0 & 0 & 0 & 0 & 0 & 1 & 0 \\
 0 & N-2 & 0 & N-3 & 0 & 0 & 0 \\
 0 & 0 & 0 & 0 & 1 & 0 & 1 \\
 N-2 & 0 & 0 & 0 & 0 & N-3 & 0 \\
 0 & 0 & 1 & 0 & 0 & 0 & 1 \\
 0 & 0 & N-3 & 0 & N-3 & 0 & N-4 \\
\end{array}
\right) }\\
\rho^{\text{anti}} = \rho^{(4)} = \Scale[0.75]{\left(
\begin{array}{ccccccc}
 0 & 0 & 0 & 0 & 1 & 0 & 0 \\
 0 & 0 & 1 & 0 & 0 & 0 & 0 \\
 0 & 0 & 0 & 0 & 0 & 1 & 1 \\
 N-2 & 0 & N-3 & 0 & 0 & 0 & 0 \\
 0 & 0 & 0 & 1 & 0 & 0 & 1 \\
 0 & N-2 & 0 & 0 & N-3 & 0 & 0 \\
 0 & 0 & 0 & N-3 & 0 & N-3 & N-4 \\
\end{array}
\right) }\\
\rho^{\text{conv}} = \rho^{(5)} = \Scale[0.75]{\left(
\begin{array}{ccccccc}
 0 & 0 & 0 & 0 & 0 & 1 & 0 \\
 0 & 0 & 0 & 1 & 0 & 0 & 0 \\
 0 & 0 & 0 & 0 & 1 & 0 & 1 \\
 0 & N-2 & 0 & N-3 & 0 & 0 & 0 \\
 0 & 0 & 1 & 0 & 0 & 0 & 1 \\
 N-2 & 0 & 0 & 0 & 0 & N-3 & 0 \\
 0 & 0 & N-3 & 0 & N-3 & 0 & N-4 \\
\end{array}
\right)}
\end{gather*}

\begin{align*}
\rho^{\text{disj}} = \rho^{(6)} = \Scale[0.65]{\left(
\begin{array}{ccccccc}
 0 & 0 & 0 & 0 & 0 & 0 & 1 \\
 0 & 0 & 0 & 0 & 0 & 0 & 1 \\
 0 & 0 & 0 & 0 & N-3 & N-3 & N-4 \\
 0 & 0 & 0 & 0 & N-3 & N-3 & N-4 \\
 0 & 0 & N-3 & N-3 & 0 & 0 & N-4 \\
 0 & 0 & N-3 & N-3 & 0 & 0 & N-4 \\
 (N-3) (N-2) & (N-3) (N-2) & (N-4) (N-3) & (N-4) (N-3) & (N-4) (N-3) & (N-4) (N-3) & (N-5) (N-4)\\ 
\end{array}
\right)
}
\end{align*}
The algebraic identities satisfied by $\{\R^{(k)}\}_{k=0}^6$ can be read off from rows of elements of the intersection
algebra. For instance the fifth row of $\rho^{\text{chain}}$ is
$(N-2,0,0,0,0,N-3,0)$. The fifth relation is the Anti-Chain relation,
so this row implies  that $\R^{\text{chain}}\R^{\text{anti}} = (N-2)
\R^{\text{id}} + (N-3) \R^{\text{conv}}$.

In the context of network generation one can assume a symmetric
covariance matrix, so the
coefficients of the chain and anti-chain relations are assumed equal. The
eigenvalues of $\bs(\ba)$ can be found by computing the eigenvalues of
the corresponding $7 \times 7$ matrix $\rho(\bs(\ba))$. This can be done
analytically, and obviously implies that $\bs$ can have no more than $7$
distinct eigenvalues. In fact there are only five distinct
eigenvalues, as $\lambda_{4/5}$ have multiplicity two. It is still
somewhat surprising that the eigenvalues have such simple dependence
of $\ba$: linear or algebraic of degree two. There is an algebraic reason for this but it is
somewhat involved, and not terribly important to the problem of random
network generation, so we will not address this issue in the current
paper.
  
\end{example}


\subsection{Numerical Implications for the SONETS method}

The fact that there is a coherent configuration underlying the SONETS algorithm has some implications for numerically implementing the method. There are two basic issues to be addressed. The first is understanding the region in $\ba$ space where $\bs(\ba)$ is positive definite and represents a valid covariance. The second is finding an efficient way to compute $\bs^{\frac12}(\ba).$ Of course one can always do this using the spectral theorem for symmetric matrices, but for large networks this can be costly. The first problem has essentially been solved, as we have given formula for the eigenvalues as a function of $\ba$ valid for any $N$. In this section we show how to compute $\bs^{\frac12}(\ba)$ very efficiently - in time independent of the size of the network. 

We would like to emphasize at this point that the work of Nykamp and Zhao, particularly Zhao's thesis\cite{Zhao.2012}, presaged a lot of the ideas in this section. While they did not have the full algebraic structure of the problem Zhao gives an asymptotic expansion (for $N$ large) for computing $\bs^{\frac12}(\ba)$ which is closely related to what we will talk about in this section. What we show in this section is that, for basically the same amount of computation, one can actually compute $\bs^{\frac12}(\ba)$ exactly. 

In the  interest of simplicity we first discuss the extraction of the square root for the Johnson scheme, which is simpler for several reasons, and then we will discuss the problem for the directed scheme. 

\subsection{The Johnson scheme and undirected graph generation.}
For the Johnson scheme the analogous problem would be as follows. We have a covariance matrix of the form 
\begin{align*}
\bs(\ba) = I + \alpha_1 \R^{(1)} + \alpha_2 \R^{(2)}
\end{align*}
for which we would like to find the matrix square root. One way to do this would be to find $\beta_{0,1,2}$ such that 
\begin{align*}
(\beta_{0} \I + \beta_{1} \R^{(1)} + \beta_2 \R^{(2)})^2 =  \I + \alpha_1 \R^{(1)} + \alpha_2 \R^{(2)}. 
\end{align*}
Using the fact that fact that the $R^{(k)}$ form a closed algebra this is equivalent to the following system of coupled quadratic equations:
\begin{align}
\beta_0^2 + 2 (N-2) \beta_1^2 + \binom{N-2}{2} \beta_2^2 &= 1, \label{eqn:nlm1}
\\
(N-2) \beta_1^2 + 2 \beta_0 \beta_1 + 2(N-3) \beta_1\beta_2 + \binom{N-3}{2} \beta_2^2 &= \alpha_1, \label{eqn:nlm2}
\\
4 \beta_1^2 + 2 \beta_0\beta_2 + 4 (N-4) \beta_1 \beta_2 + \binom{N-4}{2} \beta_2^2 &= \alpha_2. \label{eqn:nlm3}
\end{align}
At this point we should say a word about multiplicity of
solutions. The covariance matrix $\bs$ is an
$\binom{N}{2}\times\binom{N}{2}$ matrix, so there are in principle
$2^{\binom{N}{2}}$ possible matrix square roots - there is one sign
choice for each eigenvalue. We are only interested in one, say the
unique positive definite one. The equations
(\ref{eqn:nlm1}--\ref{eqn:nlm3}) will have eight real solutions if
$\alpha_{1,2}$ are such that $\bs$ is positive definite. This is
because, by assuming that $\bs^{\frac12}(\ba)$ lies in the  algebra, we are forcing the square root to have three invariant subspaces, so there is one sign choice for each invariant subspace. 

Rather than attempt to solve the system of coupled quadratic equations (\ref{eqn:nlm1}--\ref{eqn:nlm3}) directly  it is easier to work through the eigenvalues. Since there are three distinct eigenspaces if all of the eigenvalues agree the matrices must agree. At the level of eigenvalues, using (\ref{eqn:lambda0}--\ref{eqn:lambda2}), we have the equations 
\begin{align}
\left(\beta_0+2 (N-2)\beta_1 + \binom{N-2}{2} \beta_2\right)^2 &= 1 +2 (N-2)\alpha_1 + \binom{N-2}{2} \alpha_2,
\\
\left(\beta_0+ (N-4)\beta_1 - (N-3) \beta_2\right)^2 &= 1 + (N-4)\alpha_1 - (N-3) \alpha_2,
\\
\left(\beta_0-2 \beta_1 + \beta_2\right)^2 &= 1 - 2 \alpha_1+\alpha_2,
\end{align}
which is equivalent to 
\begin{align*}
\begin{pmatrix}
1 & 2(N-2)& \binom{N-2}{2}
\\
1 & N-4 & -N+3
\\
1 & -2 & 1
\end{pmatrix}
\begin{pmatrix}
\beta_0
\\
\beta_1
\\
\beta_2
\end{pmatrix}
=
\begin{pmatrix}
\pm \sqrt{1 + 2(N-2)\alpha_1 + \binom{N-2}{2} \alpha_2 }
\\
\pm \sqrt{1 + (N-4)\alpha_1 - (N-3) \alpha_2}
\\
\pm \sqrt{1 - 2 \alpha_1+\alpha_2}
\end{pmatrix}
\end{align*}
where the $\pm$ are independent. The matrix is invertible for $N > 2$ so this gives an explicit formula for the eight roots of (\ref{eqn:nlm1}--\ref{eqn:nlm3}).

Having done this it is straightforward to implement the SONETS
procedure for an undirected graph. We compute $\bs^{\frac12}(\ba)\bo$,
an then perform thresholding on $\bs^{\frac12}(\ba)\bo$ to determine
the presence or absence of a particular edge. In this example, and
most of the numerical examples in this paper the graphs have $N=100$
vertices and $\alpha_0=1$, and the thresholding level $x$ is chosen to make the probability
of a single edge ${\mathbb P}(\omega_i > x)=.1$. The other two parameters are chosen
to satisfy $2(N-2) \alpha_1 + \binom{N-2}{2} \alpha_2=0.$ A word about
this choice is in order. The covariance matrices considered here
always have the vector $(1,1,1,\ldots,1)^t$ as an eigenvector: in this
particular case the corresponding eigenvalue is $\alpha_0 + 2 (N-2) \alpha_1
+ \binom{N-2}{2} \alpha_2.$ If each $\omega_i$ has unit variance then
$\bo . (1,1,1,\ldots,1)$ is, by the central limit theorem, typically of the
order of $\sqrt{N}$. If  $\alpha_0 + 2 (N-2) \alpha_1
+ \binom{N-2}{2} \alpha_2$ is large (meaning much larger than $O(1)$)
then the mean of $\bs^{\frac12}(\ba)$  will typically be much larger than $O(\sqrt{N}).$ What happens in this case
is that realizations of the graph will tend to either have very few edges, or
very many edges, and the vertex degree distribution of a single
realization will look nothing like the average distribution. For this
reason in all numerics examples we choose the coefficients so that the
eigenvalue of $\bs$  corresponding to the vector  $ (1,1,1,\ldots,1)$
is $O(1)$. 
 
In Figure \ref{fig:johnson} we give the results of some numerical
simulations for $\alpha_1=\{0.0,~0.2,~0.4\}$. The left-hand panels depict a single realization of the
random graph. The graph is drawn with vertices of high degree located
most centrally, while those of the lowest degree lie near the periphery. The right-hand panels give a histogram of the
vertex degrees from an ensemble of one hundred realizations of the random
graph, with a dark square denoting the sample mean vertex degree. Since the single
edge probability is $p=0.1$ the expected value of the vertex degree
is $9.9$, and the sample mean is quite close to this value in all of the
experiments.   The case $\alpha_1=0$ corresponds exactly to the
Erd\H{o}s-R\'enyi case. Here we see a roughly normal distribution of
the vertex degree about the mean. Recall that  $\alpha_1$ measures correlations between
edges that share a vertex. Increasing  this correlation coefficient to
$\alpha_1 = 0.2$ we induce correlations between edges sharing a
vertex. This produces a dramatic broadening of the distribution: there
are many more vertices of high degree as well as many more vertices of low degree.  This
becomes even more pronounced as the coefficent is further increased to  $\alpha_1=0.4$.

\begin{figure}[H]
\centering
\begin{tabular}{|c|c|}
\hline
\includegraphics[width=7.5cm,height=5.5cm,keepaspectratio]{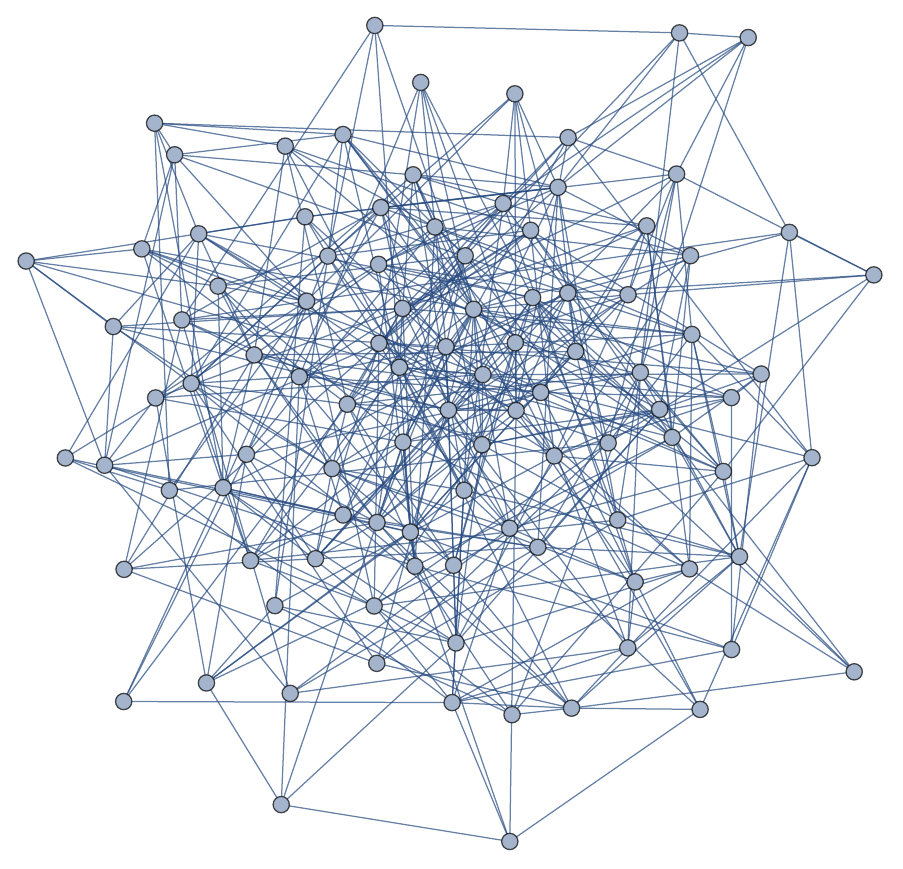} &
\includegraphics[width=7.5cm,height=5.5cm,keepaspectratio]{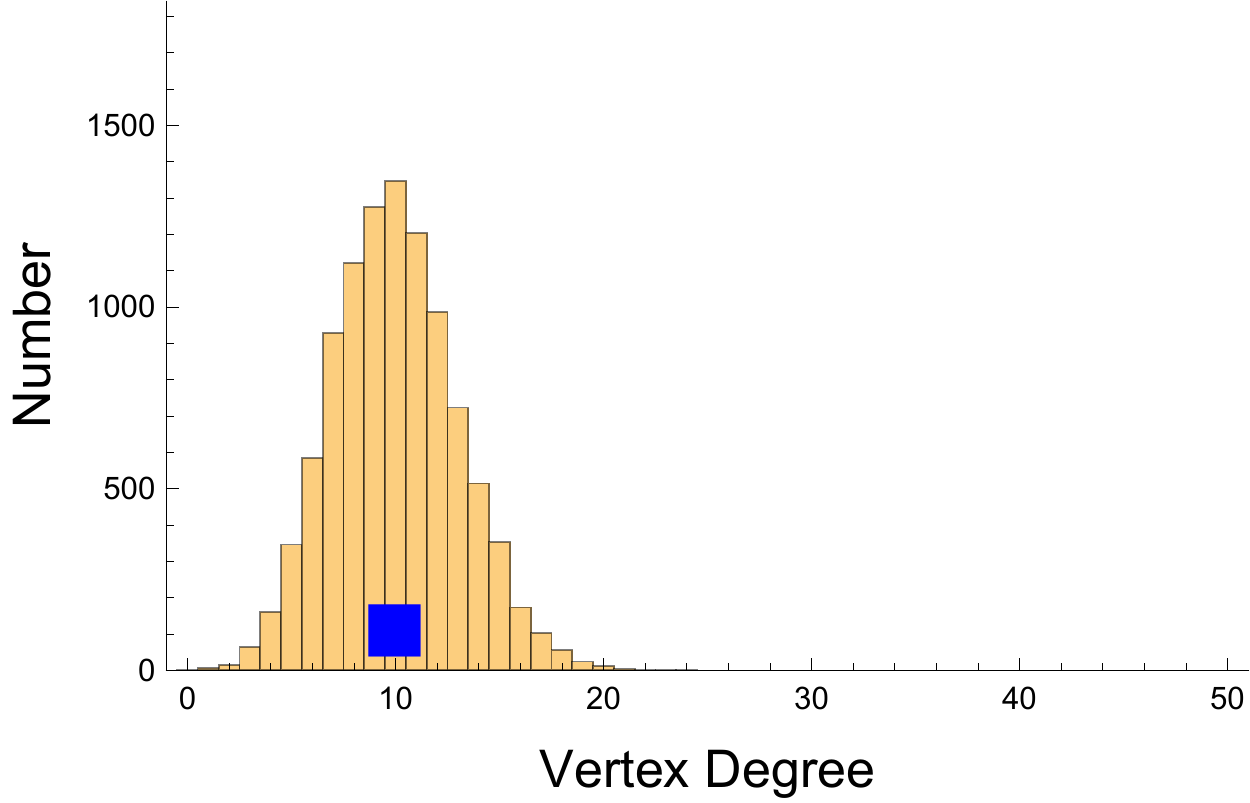}
\\
\hline 
\includegraphics[width=7.5cm,height=5.5cm,keepaspectratio]{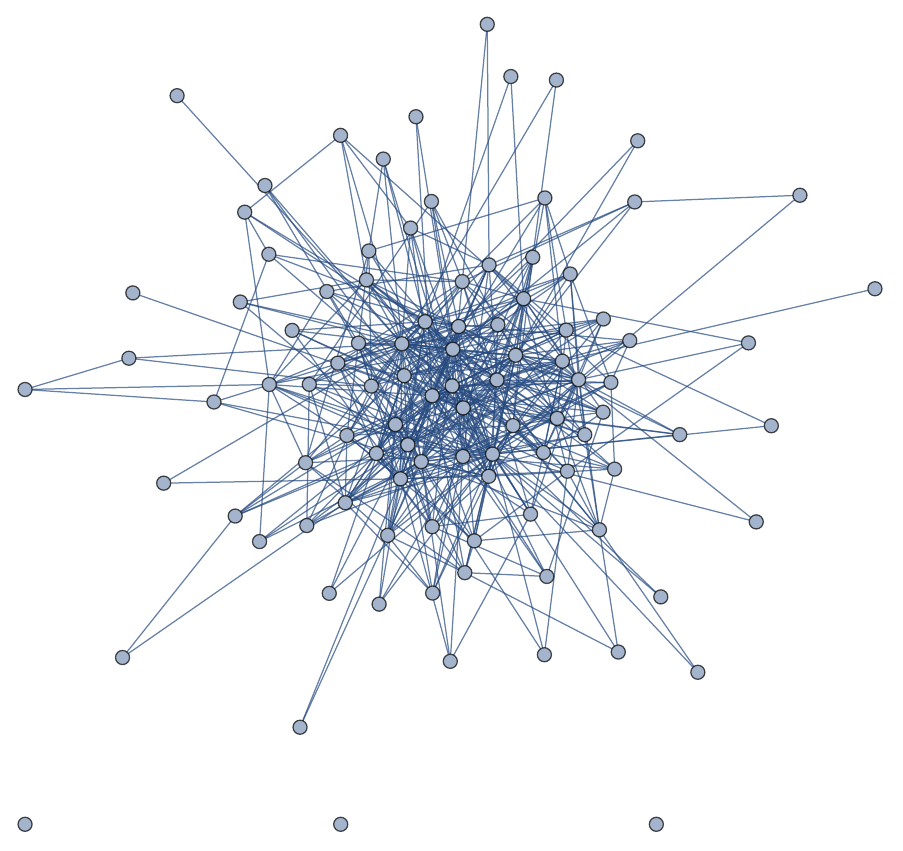} & \includegraphics[width=7.5cm,height=5.5cm,keepaspectratio]{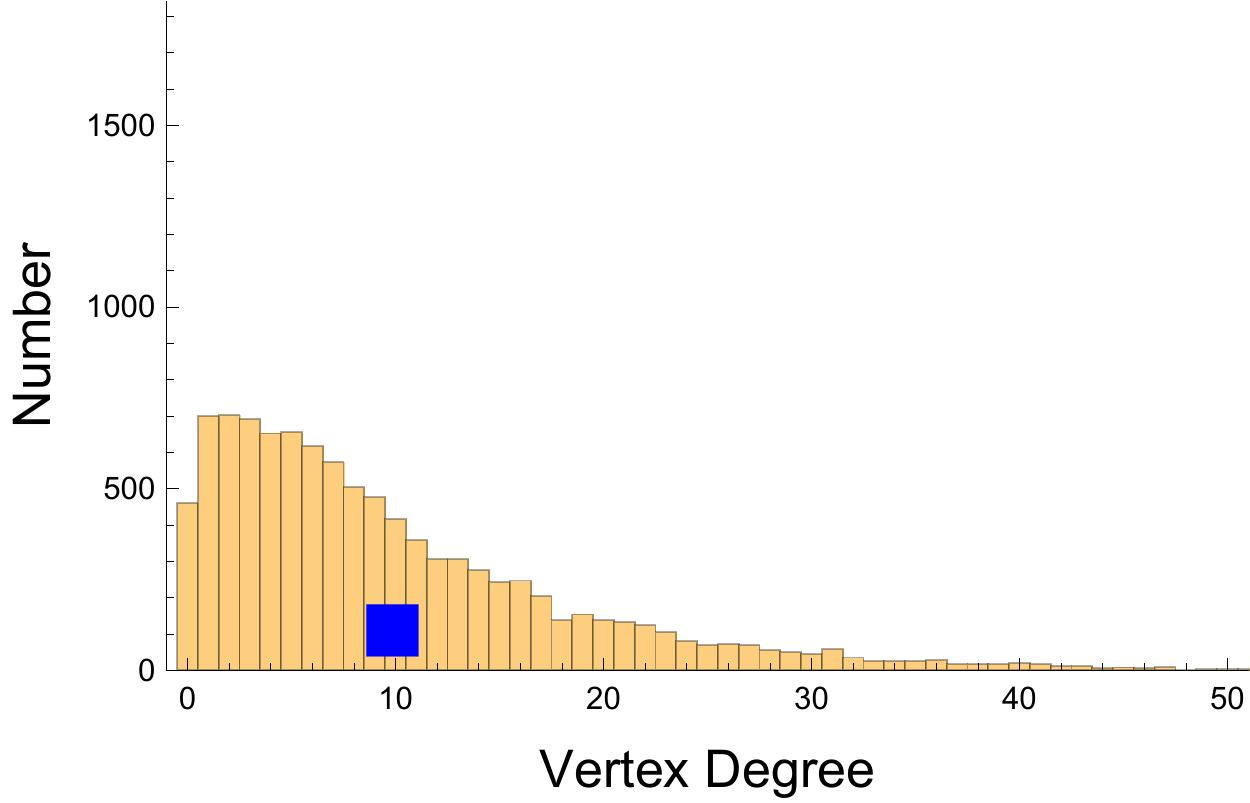}
\\
\hline
\includegraphics[width=7.5cm,height=5.5cm,keepaspectratio]{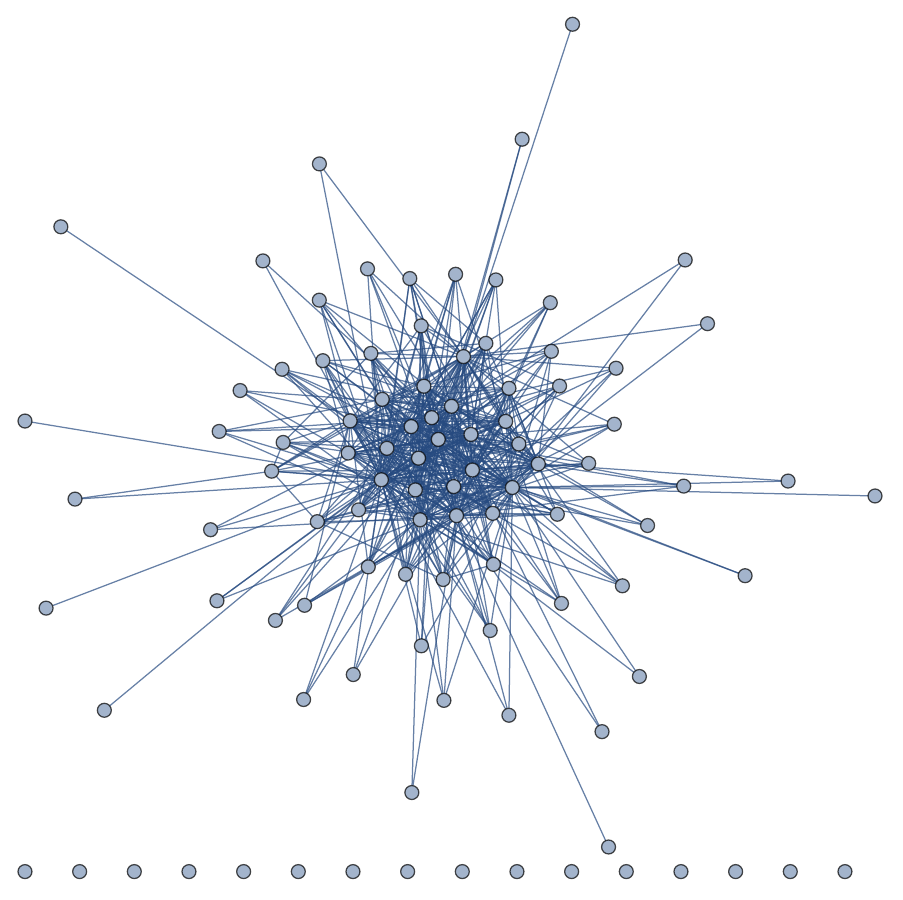} & \includegraphics[width=7.5cm,height=5.5cm,keepaspectratio]{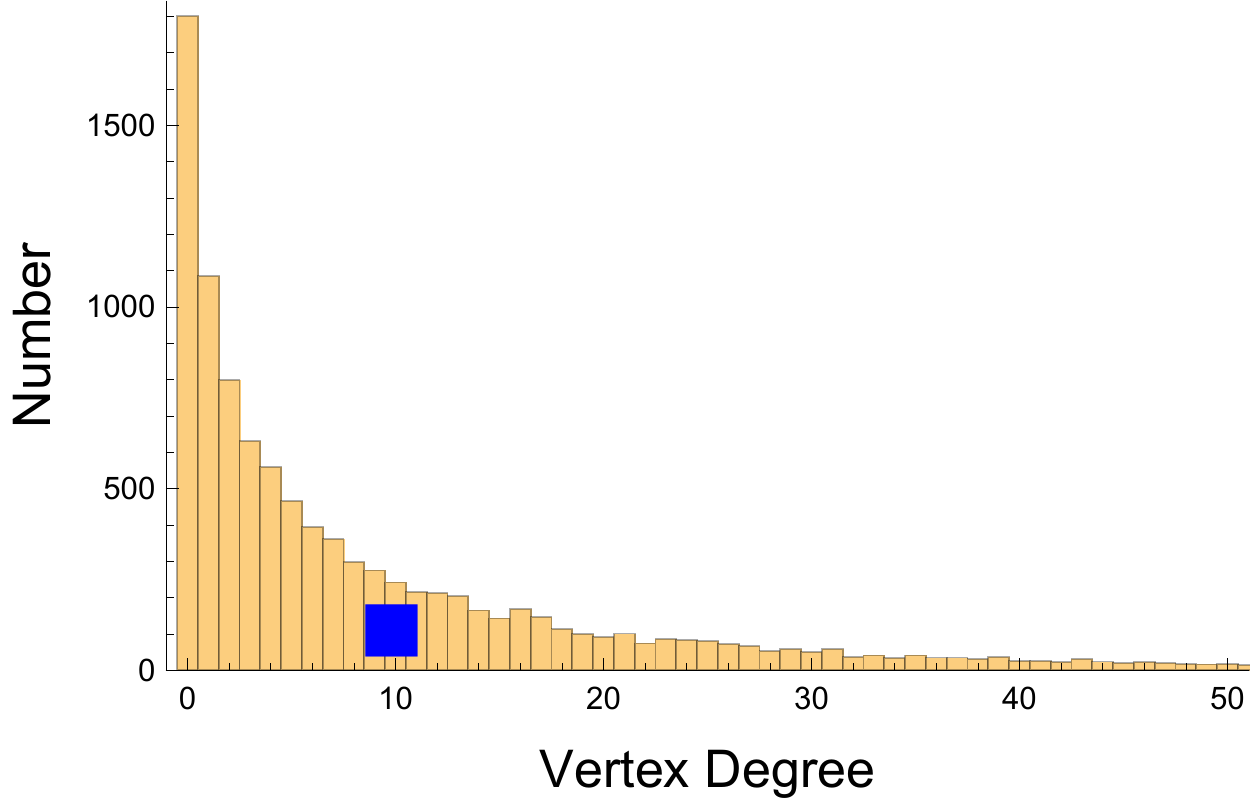}
\\
\hline
Centrality Graph  &
Histogram of Vertex Degree \\
\hline
\end{tabular}
\caption{Random graphs ($N = 100$ vertices) generated with the
  undirected analog of the Nykamp-Zhao algorithm. The parameters were
  chosen so that $\alpha_0 = 1$ and $2(N-2) \alpha_1 +
  \binom{N-2}{2} \alpha_2 = 0$. The lefthand panels depict one
  realization of the random graph, with the graph drawn so that higher
degree vertices are located more centrally and lower degree vertices
are further out, with unconnected vertices arranged along the bottom
edge. The righthand panels represent histograms of the vertex degree
drawn from 100 realizations of the random graph. The parameter values
are (to to bottom) $\alpha_1 = 0.0, 0.2, 0.4 $.}
\label{fig:johnson}
\end{figure}

\subsection{The Nykamp-Zhao HCC and random directed graph generation.}

There are a couple of minor complications when moving from the Johnson association scheme to the Nykamp-Zhao HCC. It is straightforward to compute that, owing to the algebra satisfied by the Nykamp-Zhao HCC we have the following identity
\begin{gather*} 
(\beta_0 \R^{\text{id}} + \beta_1 \R^{\text{recip}} + \beta_2  \R^{\text{div}} + \beta_3 \R^{\text{chain}} + \beta_3
\R^{\text{anti}} + \beta_4 \R^{\text{conv}} + \beta_5 \R^{\text{disj}})^2 =
\\
\alpha_0 \R^{\text{id}} + \alpha_1 \R^{\text{recip}} + \alpha_2  \R^{\text{div}} + \alpha_3 \R^{\text{chain}} + \alpha_3
\R^{\text{anti}} + \alpha_4 \R^{\text{conv}} + \alpha_5 \R^{\text{disj}}
\end{gather*}
where the coefficients $\{\alpha_i\}_{i=0}^5$ and
$\{\beta_i\}_{i=0}^5$ are related through 
\begin{align} \label{eqn:SquareRootMap}
\alpha_0 &= \beta_0^2 + \beta_1^2 + (N-2) (\beta_2^2 + 2 \beta_3^2 + \beta_4^2) + (N-3)(N-2) \beta_5^2,
\\
\alpha_1 &= 2 \beta_0 \beta_1 + (N-2) (2 \beta_2 \beta_3 + 2 \beta_3 \beta_4) + (N-2)(N-3) \beta_5^2,
\\
\alpha_2 &= 2 \beta_0 \beta_2 + 2 \beta_1 \beta_3 +(N- 3) \beta_3^2 + 2\beta_3 \beta_4 +2(N- 3) \beta_3\beta_5
\\
&+ 2(N- 3) \beta_4 \beta_5 + (N-3(N-4)) \beta_5^2 + (N-3) \beta_2^2, \nonumber
\\
\alpha_3 &= \beta_1 \beta_2 + 2 \beta_0 \beta_3 + (N-3) \beta_2\beta_3 + a_3^2 + \beta_1\beta_4 + \beta_2 \beta_4 +  (N-3) \beta_3 \beta_4
\\
&+ (N-3) \beta_2 \beta_5 + 2(N-3) \beta_3 \beta_5 + (N-3) \beta_4 \beta_5 + (N-3)(N-4) \beta_5^2, \nonumber
\\
\alpha_4 &= 2 \beta_1 \beta_3 + 2 \beta_2 \beta_3 + (N-3) \beta_3^2 + 2 \beta_0\beta_4 + (N-3) \beta_4^2
\\
&+ 2(N-3) \beta_2 \beta_5 + 2(N-3) \beta_3 \beta_5 + (N-3)(N-4) \beta_5^2, \nonumber
\\
\alpha_5 &= 2 \beta_2 \beta_3 + 2 \beta_3^2 + 2 \beta_2 \beta_4 + 2 \beta_3 \beta_4 + 2 \beta_0 \beta_5 + 2 \beta_1 \beta_5 + 2(N-4) \beta_2 \beta_5
\\
&+ 4(N-4) \beta_3 \beta_5 + 2(N-4) \beta_4 \beta_5 + (N-4)(N-5) \beta_5^2. \nonumber
\end{align}
Here we have assumed the symmetric case, where the coefficients of the chain and anti-chain motifs are the same. Note that similar equations were  derived by Zhao in his thesis work: on pages 94 and 95 in Appendix A of Zhao's thesis the first five of these equations appear (with $\beta_5=0$) as (A.3)-(A.7). The reason $\beta_5=0$ is because Nykamp and Zhao do not consider the disjoint motif. Note that this implies that the algorithm, as presented by Zhao, induces correlations among disjoint edges, as $\alpha_5$ is typically not zero when $\beta_5=0$. The algorithm presented in Zhao's thesis presents an approximate method for solving the above coupled system of quadratic equations by assuming $\beta_{\text{div}/\text{chain}/\text{conv}}=O(\frac{1}{\sqrt{N}})$ and relating the resulting reduced system to the numerical extraction of the square root of a $2 \times 2$ symmetric matrix. The main point of this section is to observe that one can, for essentially the same computational work, solve this system exactly rather than approximately. 

In the previous Johnson scheme calculation there were three eigenvalues that depended linearly on the three coefficients defining an element of the algebra. This allowed us the ``diagonalize'' the solution of the coupled quadratic equations. In the Nykamp-Zhao HCC there are only five distinct eigenvalues, and seven parameters that define an element of the algebra (six if we are assuming a symmetric matrix), so obviously the eigenvalues cannot be used as coordinates in the same way. Further the map between the parameters of the algebra and the eigenvalues is not linear, so the inversion of the map becomes more complicated. However there is still a simple way to extract the square root of a linear combination of the matrices $\R^{(k)}$ in computational time which is independent of the size of the graph. To begin we note a couple of facts:
\begin{itemize}
\item The seven matrices $\R^{(k)}$ and the corresponding matrices $\rho^{(k)}$ are linearly independent as vectors in $\RRR^{N(N-1) \times N(N-1)}$ and $\RRR^{49}$ respectively.
\item The map between the algebra spanned by $\R^{(k)}$ and that
spanned by $\rho^{(k)}$ is therefore invertible.
\item Given $\rho = \sum_k \alpha_k \rho^{(k)}$ in the intersection algebra one can recover the coefficients $\alpha_k$ by solving the linear system
\begin{align*}
\G \ba = \bb
\end{align*}
where $\beta_k = \tr(\rho^\top \rho^{(k)})$ and 
\begin{align*}
\Scale[0.45]{\G = \left(
\begin{array}{ccccccc}
 7 & 1 & 3 N-10 & N-4 & N-4 & 3 N-10 & N^2 - 9N + 20
 \\
 1 & 7 & N-4 & 3 N-10 & 3 N-10 & N-4 & N^2- 9N+20
 \\
 3 N-10 & N-4 & 7 N^2-40 N+66 & N^2-8 N+18 & 3 N^2-20 N+34 & N^2-8 N+16 & 3 N^3-33N^2+126 N-166
   \\
 N-4 & 3 N-10 & N^2-8 N+18 & 7 N^2-40 N+66 & N^2-8 N+16 & 3 N^2-20 N+34 & 3N^3-33
   N^2+126 N-166
   \\
 N-4 & 3 N-10 & 3 N^2-20 N+34 & N^2-8 N+16 & 7 N^2-40 N+66 & N^2-8 N+18 & 3 N^3-33
   N^2+126 N-166
   \\
 3 N-10 & N-4 & N^2-8 N+16 & 3 N^2-20 N+34 & N^2-8 N+18 & 7 N^2-40 N+66 & 3 N^3-33
   N^2+126 N-166
   \\
 N^2-9 N+20 & N^2-9 N+20 & 3 N^3-33 N^2+126 N-166 & 3 N^3-33 N^2+126 N-166 & 3 N^3-33
   N^2+126 N-166 & 3 N^3-33 N^2+126 N-166 & 7 N^4-94 N^3+499 N^2-1232 N+1186
\end{array}
\right)}
\end{align*}
is the Gram matrix of $\{\rho^{(k)}\}_k$: $\G_{ij} = \tr((\rho^{(i)})^\top \rho^{(j)}).$ 
\end{itemize}   
It is easy to check that the Gram matrix $\G$ has a determinant that is
a polynomial of degree $12$ in $N$ with no real roots, and is thus
always invertible. Therefore an efficient way to compute the (positive definite) square root of a positive definite covariance matrix $\bs(\ba) = \sum_k \alpha_k \R^{(k)}$ is as follows: 
\begin{enumerate}
\item Compute the image of $\bs$ under the homomorphism: $\varsigma =
  \rho(\bs)=\rho(\sum_k \alpha_k \R^{(k)}$.
\item Compute the square root of $\varsigma$ spectrally in the usual
  way: $\varsigma^{\frac12} = \U \bl^{\frac12} \U^{-1}$, where
  $\varsigma = \U \bl \U^{-1}$. While non-symmetric $\varsigma$ is
  diagonalizable with positive eigenvalues.
\item Compute the coefficients $\beta_k$ in the decomposition of $\varsigma^{\frac12} = \sum_k \beta_k \rho^{(k)}$ by solving $\sum_j \G_{ij} \beta_j = \tr((\varsigma^{\frac12} )^\top \rho^{(i)})$.
\item Pull back to the covariance via $\bs^{\frac12} = \sum_k \beta_k \R^{(k)}$.   
\end{enumerate}
We will generally assume that positive definite square root is the one
taken, although any of the $2^5 = 32$ branches of the square roots are
compatible with the algebra may be chosen. As a final practical note
we remark that for large networks all of the matrices $\R^{X}$ are
sparse except for $\R^{Disj}$. The matrices sum to the
matrix with entries $1$, which is rank one. It is simple to multiply any
vector ${\bf x}$ by the matrix of all ones, since it is rank one, so
one can efficiently multiply any vector by the covariance using sparse
techniques.    

We refer the reader to the original papers of Nykamp, Zhao and
collaborators for more numerical studies, but we present a few numerical simulations of graphs
generated using the SONETS scheme as presented here. As one might
expect our results are qualitatively similar. To keep things simple we
vary only two parameters, $\alpha^{\text{recip}}$ and $\alpha^{\text{conv}}$. As a
reminder $\alpha^{\text{recip}}$ induces correlations between a directed edge and the
reciprocal (oppositely directed) edge, while $\alpha^{\text{conv}}$ induces correlations
between edges that are oriented into the same vertex. Figure
(\ref{fig:nykamp}) presents some numerical experiments. As in the
previous experiments we have taken a random graph with $N=100$
vertices and a threshold chosen to give an  individual edge
probability of $p=0.1$.  The top row
shows $\alpha^{\text{recip}} = \alpha^{\text{conv}}=0.$ 
The expected number of edges with no reciprocal edge in this
case is $9900\times(.1)\times(1-.1) =891$, while the expected number
of  edges where the reciprocal edge is also present is
$9900\times(.1)\times(.1)=99$. The leftmost panel depicts a single
realization of the random graph, again drawn with the vertices of
highest degree located most centrally. This realization had $867$ single edges
and $98$ reciprocal edge pairs. The middle figure gives a pixel plot
of the adjacency matrix of the graph -- each block is black if the
corresponding edge is present and white if the edge is absent. The
final panel depicts a histogram of the
distribution of the in-degree of the vertices along with a smooth
curve approximating the distribution of the out-degree. As one might
expect both the in and out-degree 
are roughly normal with a mean of  $\approx 9.9$. The sample mean
in-degree and out-degree (which must be the same) are indicated by the
square and diamond respectively. Note that here we have chosen the
eigenvalue corresponding to eigenvector $(1,1,1,\ldots,1)^t$ to be
zero, not one. This implies that the case $\alpha_1=0$ is not quite
the ER case, since $\alpha_2$ is small and negative, indicating that
disjoint edges are slightly anti-correlated. Numerics on the pure ER
case looked quite similar. 

The next row shows the effect of increasing correlations between
reciprocal edges by taking $\alpha^{recip}=0.75$. The graphs and histograms look quite similar, as do the pixel
plots, though the second pixel plot appears to have more symmetry across the diagonal
than the first picture. This is born out by the table below, which
gives an average of the number of missing edges, single edges (those
where the reciprocal edge is not present) and reciprocal edge
pairs for one hundred realizations of a random graph. \begin{figure}[H]
\centering
\begin{tabular}{|c|c|c|c|}
\hline
$(\alpha^{\text{recip}},\alpha^{\text{conv}})$ & $\#$ absent edges & $\#$ edges not in reciprocal motif & $\#$ edges in reciprocal motif
\\
\hline
$(0,0)$ & 8912.49 & 887.63 & 99.88
\\
\hline
$(0.75,0)$ & 8905.86 & 485.78 & 508.36
\\
\hline
$(0,0.75)$ & 8895.67 & 903.99 & 100.34
\\
\hline
\end{tabular}
\end{figure}
While the average number of edges has not substantially changed
from the first experiment, and the histograms of in-degree and
out-degree also look quite similar, the number of reciprocal edge pairs has gone
up dramatically, from roughly one tenth of the total edges to just over half of the total edges.

The final sequence of plots depicts the case where correlations are
induced between edges incident to the same vertex by increasing
$\alpha^{conv}$, but reciprocal edges are uncorrelated
$\alpha^{recip}=0$.  For these parameter values we have many
unconnected vertices, which are not drawn. 
This should broaden the distribution of the in-degree (which is
governed by $\alpha^{conv}$ but is not expected to markedly change the
distribution of the out-degree (governed by $\alpha^{div}$). In this case all
three plots differ markedly from the first two. The distribution of
in-degrees (histogram) is much broader, though the distibution of
out-degrees (continuous curve) does
not seem to have changed appreciably.  There are many vertices of low in-degree 
together with a few vertices of very high in-degree -- substantially
higher than occured in the first two sets of experiments. In fact the
scale chosen cuts off the total number of vertices of degree zero:
there were about $5500$ over the one hundred realizations, so on
average more than half the vertices were unconnected. The pixel plot
shows strong layering, as one might expect, with the horizontal lines
representing vertices with a high in-degree and the horizontal white
stripes representing vertices with low in-degree. Increasing
$\alpha^{\text{div}}$, conversely, would lead to vertical striping, as
well as a broadening of the  distribution of out-degree.  Also note
from the previous table that reciprocal motif again occurs with the probability
that one would expect based on the assumption of independent edges.

\begin{figure}[H]
\centering
\begin{tabular}{|c|c|c|}
\hline
\includegraphics[width=5.5cm,height=4cm,keepaspectratio]{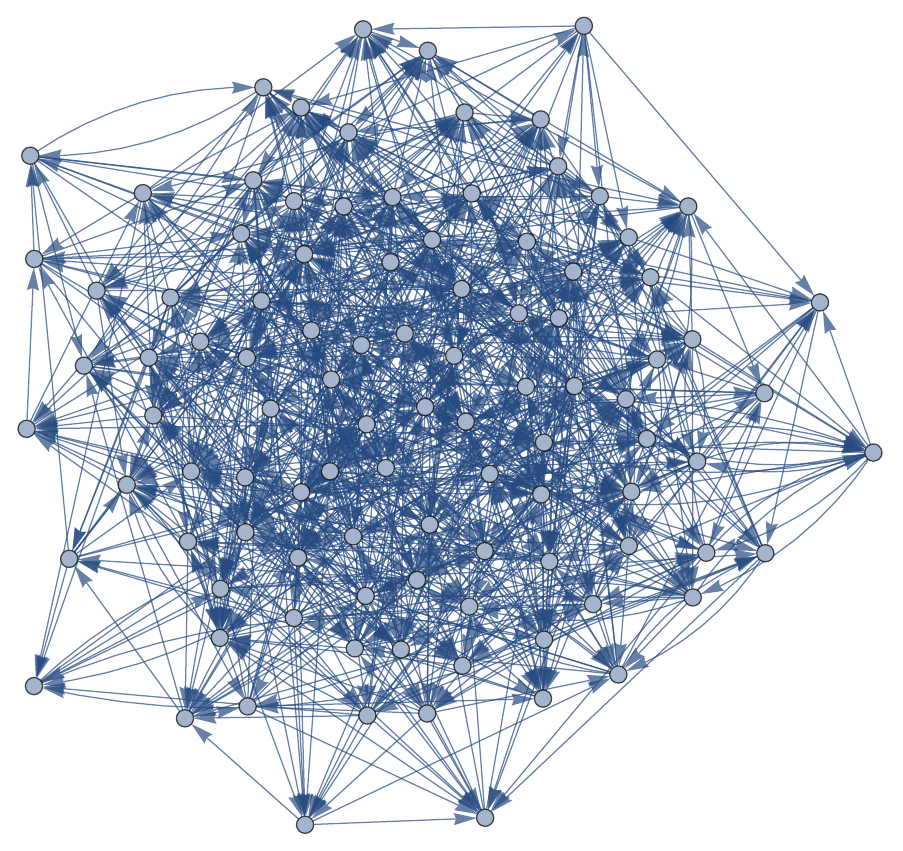} &
\includegraphics[width=5.5cm,height=4cm,keepaspectratio]{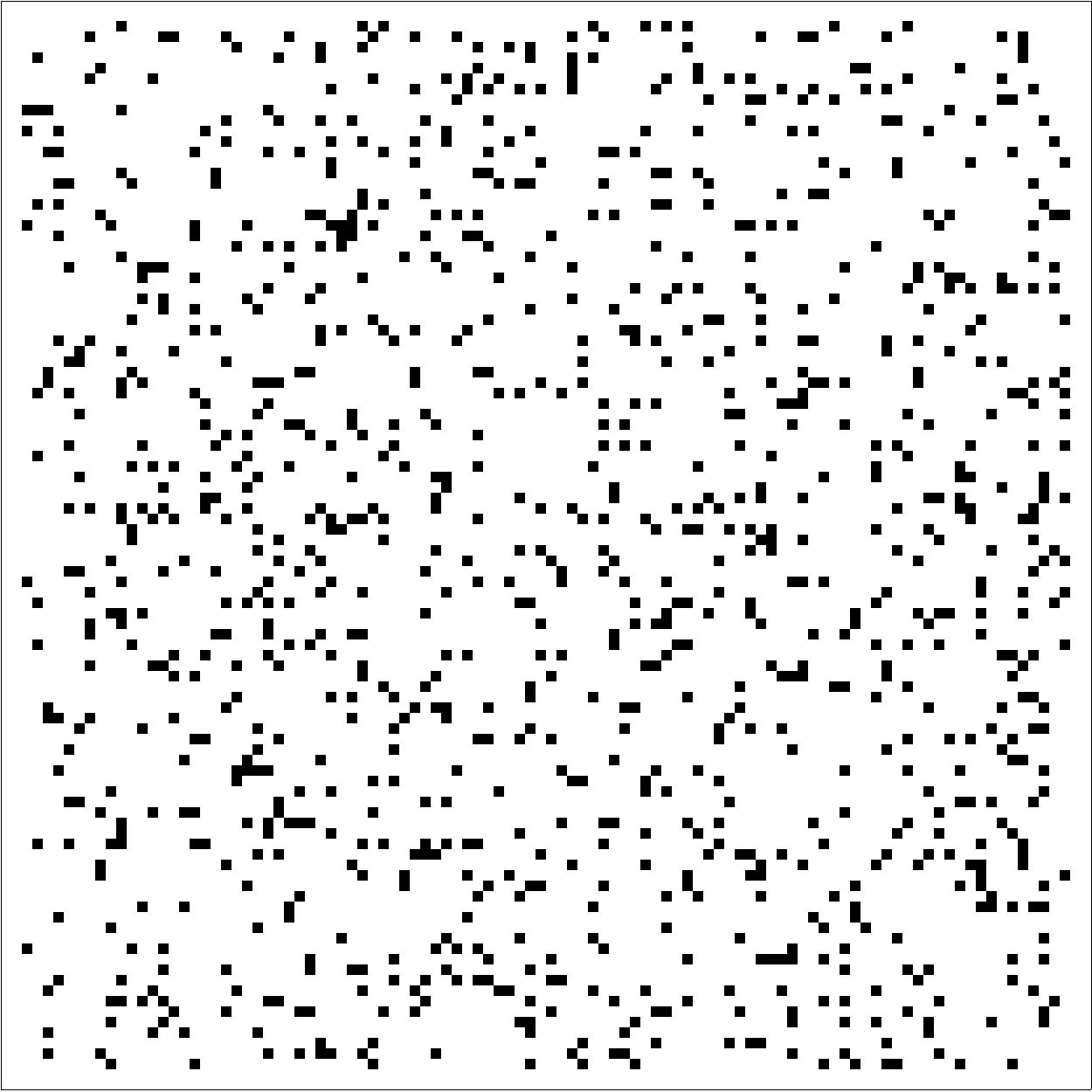}  &
\includegraphics[width=5.5cm,height=4cm,keepaspectratio]{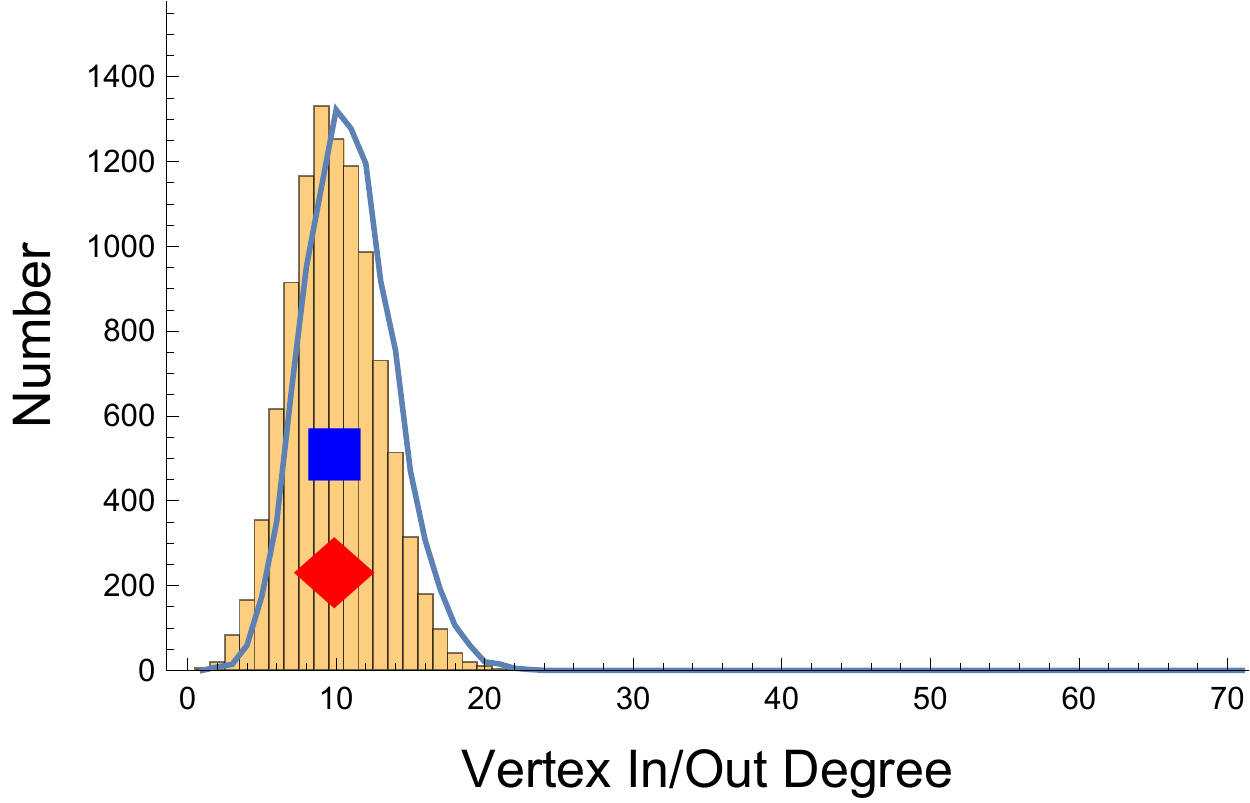}
\\

\hline
\includegraphics[width=5.5cm,height=4cm,keepaspectratio]{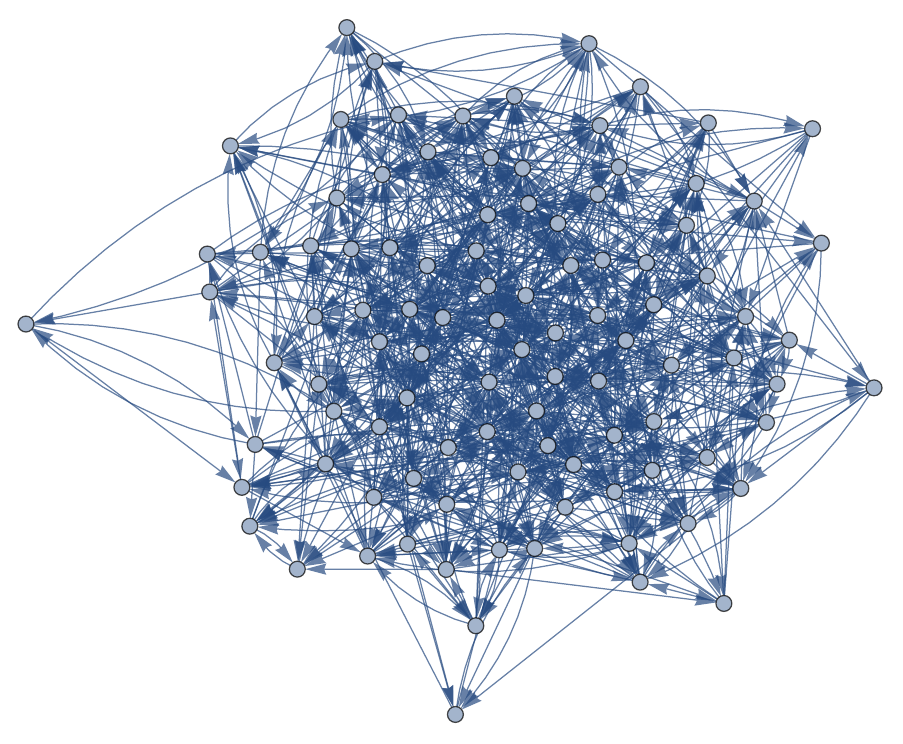} &
\includegraphics[width=5.5cm,height=4cm,keepaspectratio]{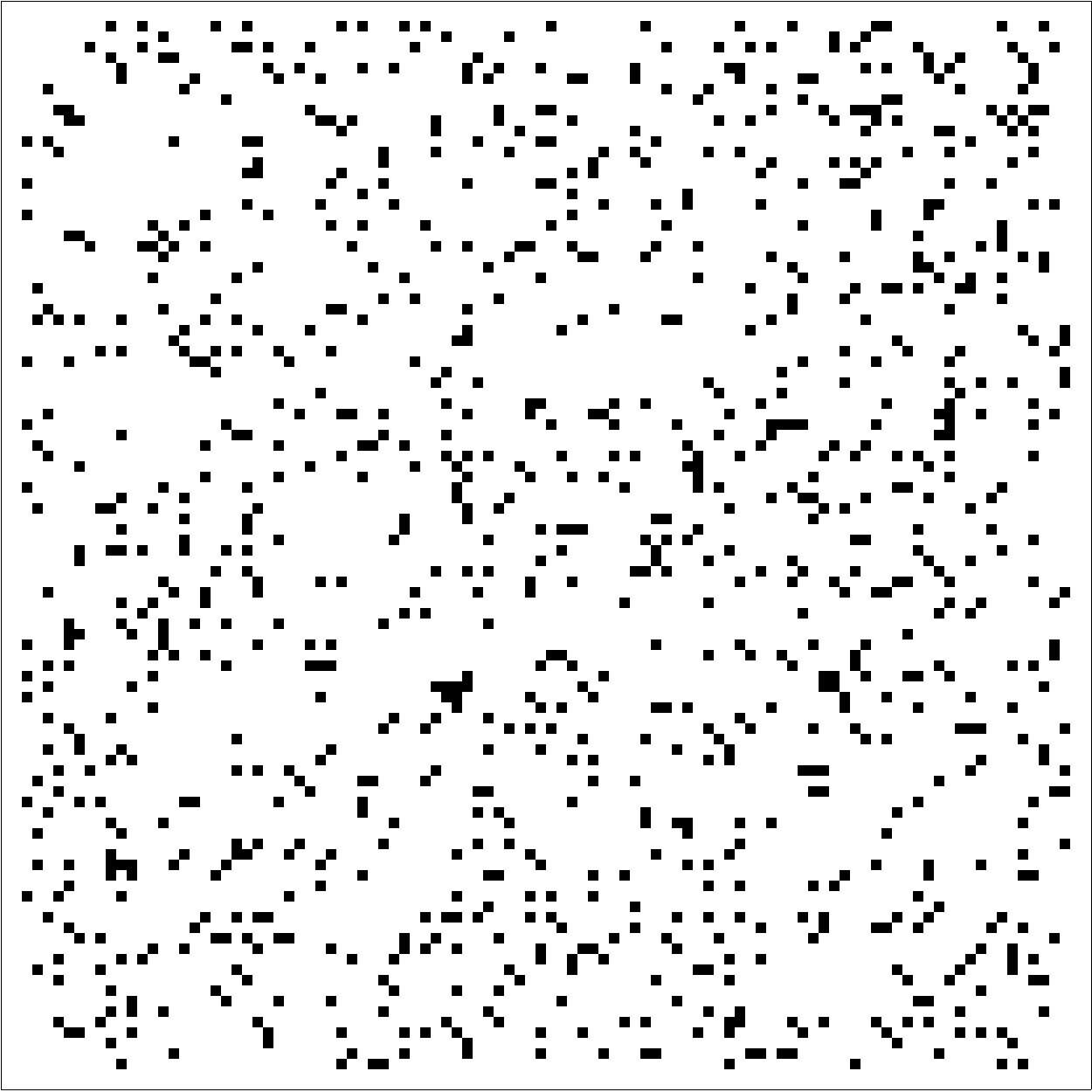} &
\includegraphics[width=5.5cm,height=4cm,keepaspectratio]{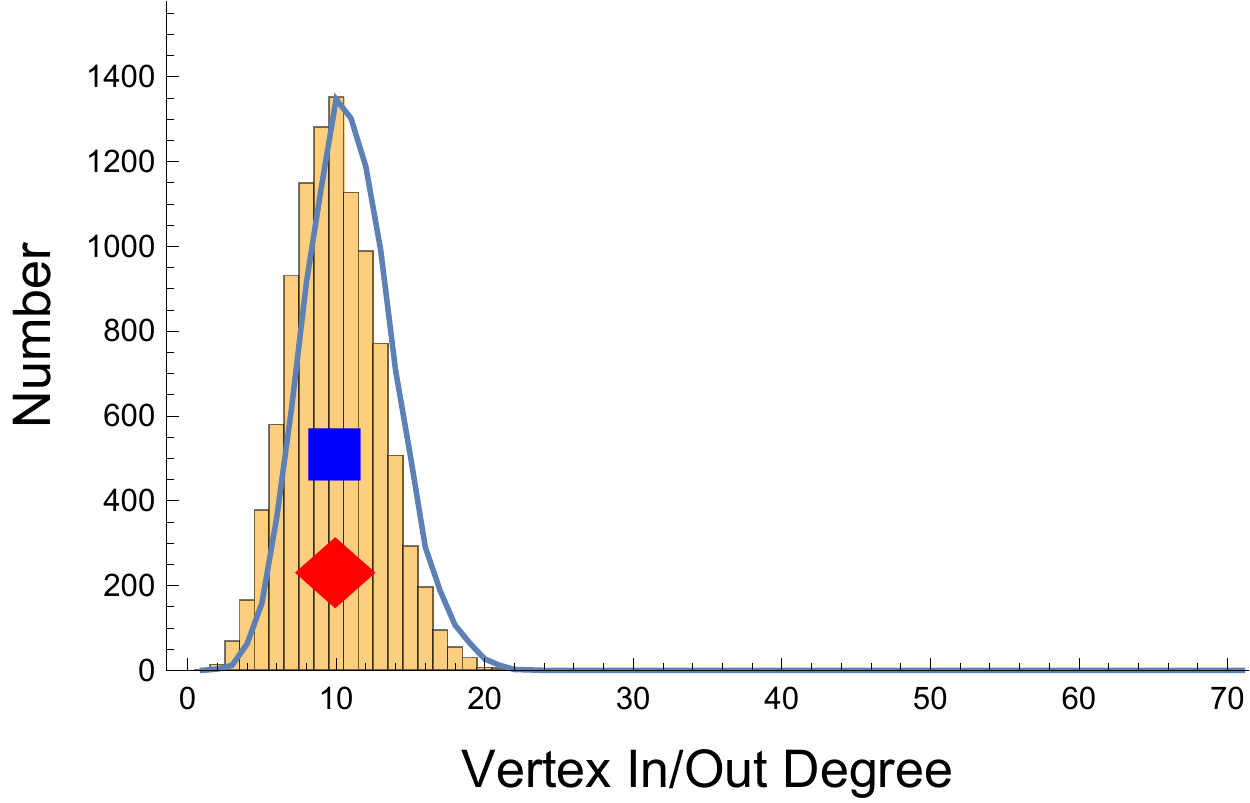}
\\

\hline
\includegraphics[width=5.5cm,height=4cm,keepaspectratio]{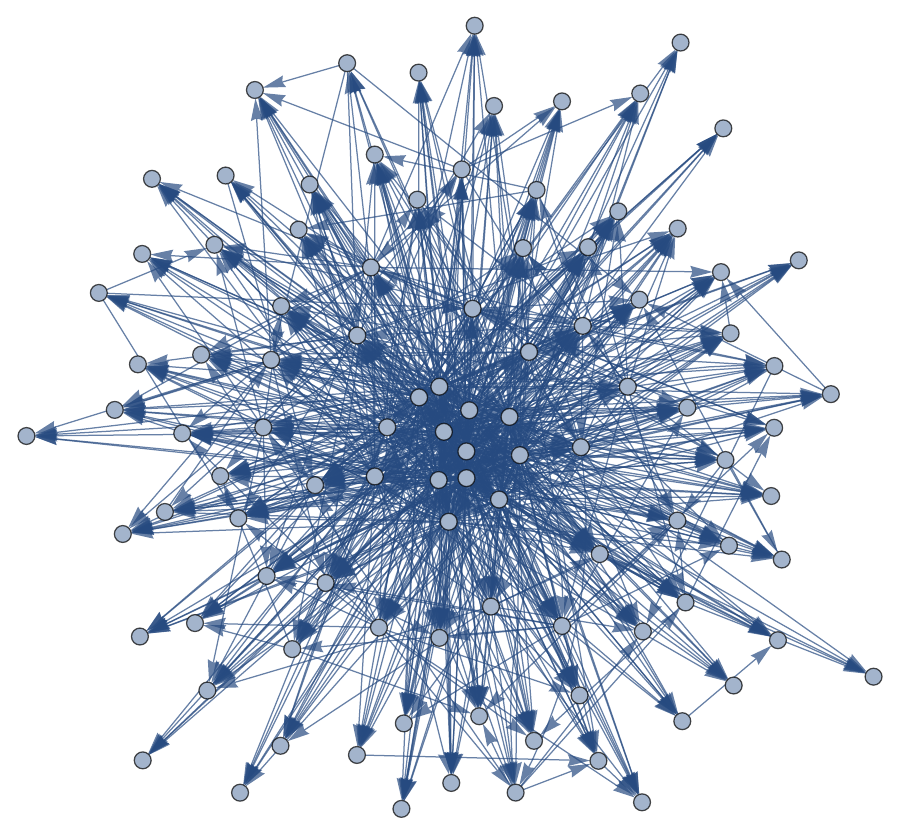} &
\includegraphics[width=5.5cm,height=4cm,keepaspectratio]{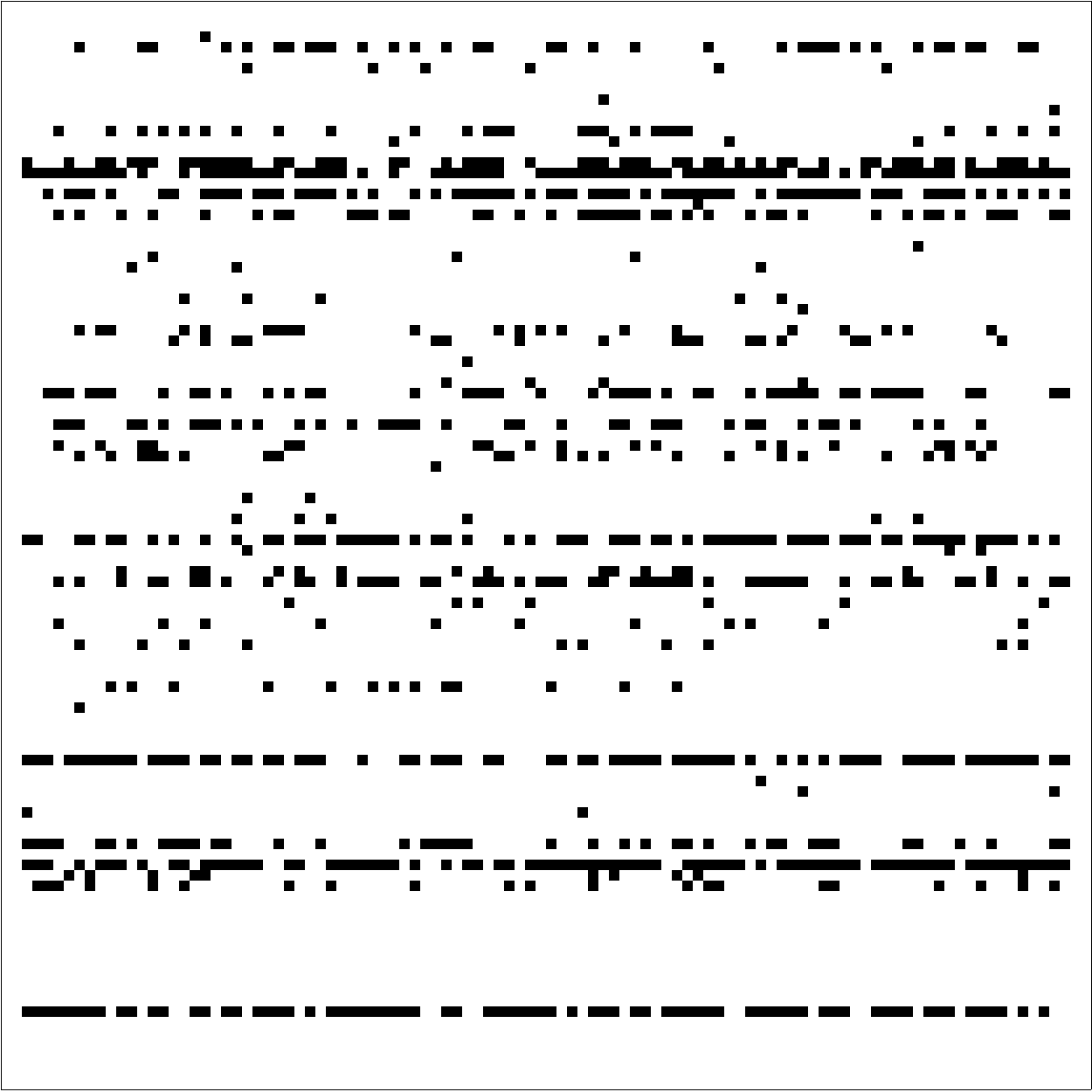} &
\includegraphics[width=5.5cm,height=4cm,keepaspectratio]{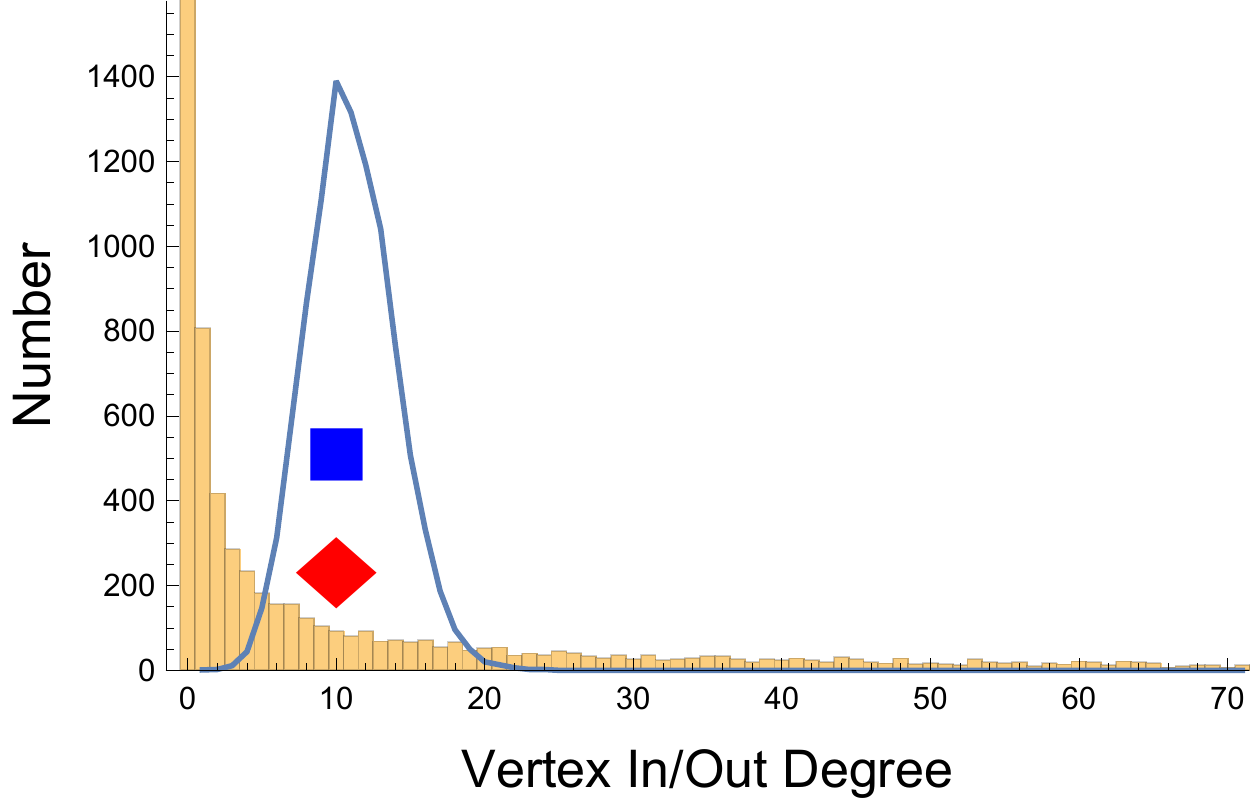}
\\

\hline
Centrality Graph & Pixel plot &
                                                              Distributions
                                                              of
                                                              vertex
                                                              in and
                     out-degree.
\\\hline
                                                          
\end{tabular}
\caption{Some numerical simulations with $N = 100$ vertices, $100$
  realizations of the random graph, and $\mathbb{P}(\omega \ge x) = 0.1$ . 
The three panels depict (left to right) a single realization of the
random graph, a pixel plot of the adjacency matrix for the graph of
that realization, and
histograms of the in-degree and out-degree of the vertices. The
parameter values are $(\alpha^{\text{recip}},\alpha^{\text{conv}}) =
(0,0)$,   $(\alpha^{\text{recip}},\alpha^{\text{conv}}) = (0.75,0)$ , $(\alpha^{\text{recip}},\alpha^{\text{conv}}) = (0,0.75)$. }
\label{fig:nykamp}
\end{figure}

\subsection{Other examples}

In addition to the Johnson scheme and the Nykamp-Zhao coherent
configuration there are many other situations where association
schemes or coherent configurations might arise in generating a random
network. The main assumption required for a coherent configuration,
assumption $(4)$ in Definition (1), is a strong homogeneity
assumption. In essence it requires that every pair satisfying a
particular relation in some sense looks like every other pair
satisfying that relation. 

One situation that might be of interest in, say, a
neuroscience context would be a network with two (or more) different
types of vertex. In the directed case one would have (with two types
of vertex) four different types of edge ($a\rightarrow a, a\rightarrow
b, b \rightarrow a, b \rightarrow b$) and numerous possible different
motifs.    

We work out the simplest possible such example here. There is a single
distinguished vertex (type $a$) and $N$ undistinguished vertices
(type $b$). Edges are undirected, and therefore of two types: those between two undistinguished vertices (undistinguished edges) and those between an undistinguished vertex and the distinguished one (distinguished edges). This leads to a (non-homogeneous) coherent configuration with nine relations. 
\begin{itemize}
\item $\R_I^{11}$ -- the identity relation for undistiguished edges. 
\item $\R_A^{11}$ -- the relation for adjacent undistinguished edges. 
\item $\R_D^{11}$ -- the relation for disjoint undistinguished edges. 
\item $\R_A^{12}$ -- two adjacent edges, the first of which is undistinguished, the second distinguished.
\item $\R_A^{21} = (\R_A^{12})^\dagger$ -- two adjacent edges, the first of which is distinguished, the second undistinguished.
\item $\R^{12}_D$ --  two disjoint edges, the first of which is undistinguished, the second distinguished.
\item $\R^{21}_D = (\R^{12}_D)^\dagger$ --  two disjoint edges, the first of which is distinguished, the second undistinguished.
\item $\R_I^{22}$ -- the identity relation for distinguished edges. 
\item $\R_A^{22}$ -- two adjacent edges, both of which are distinguished. 
\end{itemize}
The notations here are that the subscripts $I,A,D$ denote identity, adjacent and disjoint relations, meaning the edges share two, one or zero vertices. The superscripts $ij$ indicate whether the the first edge ($i$) and the second edge ($j$) are undistinguished ($1$) or distinguished ($2$). Note that this is a non-homogeneous coherent configuration, as $\R_I^{11}$ and $\R_I^{22}$ provide a partition of the diagonal. There are in principle $9^3 = 729$ structure constants, but these are very sparse due to the block structure of the relations. In particular it is easy to see that the following rules apply 
\begin{itemize}
\item $(\R_{x}^{ij})^\dagger = \R_x^{ji}$
\item $\R_x^{ij} \R_y^{kl} = 0$ if $j \neq k$
\item $\R_x^{ij} \R_y^{jk}= \sum_z \R_z^{ik}$  
\end{itemize}
where $x,y,z \in \{I,A,D\}$. There are five symmetric relations, and four transpose-conjugate pairs, so the covariance matrix has seven parameters.  The remaining quadratic relations satisfied by the algebra are given below, and define all of the non-zero intersection numbers: 
\begin{align*}
&\R_I^{11} \cdot \R_x^{ik} = \R_x^{1k} &\\
&\R_I^{11} \cdot \R_x^{2k} = \0 &\\
&\R_I^{22} \cdot \R_x^{ik} = \0&\\
&\R_I^{22} \cdot \R_x^{2k} = \R_x^{2k} &\\
&\R_A^{11}\cdot \R_A^{11} = 2 (N-2) \R_I^{11}+ (N-2) \R_A^{11} + 4
  \R_D^{11}& \\
& \R_A^{11} \cdot \R_D^{11} = (N-3) \R_A^{11} + 2 (N-4) \R_D^{11} &\\
 & \R_A^{11}\cdot \R_A^{12} = (N-2) \R_A^{12} + 2 \R_D^{12} &\\ 
&\R_A^{11} \cdot \R_D^{12} = (N-2) \R_A^{12} + 2 (N-3) \R_D^{12} & \\
& \R_D^{11} \cdot \R_A^{11} = (N-3) \R_A^{11} + 2 (N-4) \R_D^{11} & \\
& \R_D^{11} \cdot \R_A^{12} = (N-3) \R_D^{12} & \\
& \R_D^{11} \cdot \R_D^{12} = {{N-2}\choose{2}} \R_A^{12} + {{N-3}\choose{2}}
  \R_D^{12} & \\
&\R_A^{12} \cdot \R_A^{21} = 2 \R_I^{11} + \R_A^{11} &\\
&\R_A^{12} \cdot \R_D^{21} = \R_A^{11} + 2 \R_D^{11} & \\
&\R_A^{12} \cdot \R_A^{22} = \R_A^{12} + 2 \R_D^{12} & \\
&\R_D^{12} \cdot \R_A^{21} = \R_A^{11} + 2 \R_D^{11} & \\
&\R_D^{12} \cdot \R_D^{21} = (N-2) \R_I^{11} + (N-3) \R_A^{11} + (N-4)
 \R_D^{11} & \\
& \R_D^{12} \cdot \R_A^{22} = (N-2) \R_A^{12} + (N-3) \R_D^{12} .&
\end{align*}
This coherent configuration contains the Johnson scheme $J(N,2)$ as a subscheme. Since this scheme has nine relations any covariance matrix built up from this scheme will have, at most, nine distinct eigenvalues. In fact one can compute the eigenvalues using the $9 \times 9$ intersection algebra to find that there are only five distinct eigenvalues. We do not list them here but note that, similar to the Nykamp-Zhao scheme, the eigenvalues are linear or algebraic of degree two.

We generate some second order networks, again with $N=100$
vertices, for the model of a complete graph with a single
distinguished vertex. In previous cases there was a single identity element, the
coefficient of which could always be scaled to $\alpha_I=1$. In this
case since there are two types of edge the identity partitions into
two parts, which can in principle have different weights. We have chosen $\alpha_I^{22}$ (associated with
distinguished edges) to always be equal to $1$, and chosen the
threshold so that the probability of distinguished edges is $p=.1$.
The coefficient of the undistinguished edges $\alpha_I^{11}$ is chosen
to be either $\alpha_I^{11}=1$, giving a single undistinguished edge probability of
$.1$, or $ \alpha_I^{11}=2$ giving a single edge probability of about
$.183$.  In these experiments we always pick
$\alpha_A^{12}=\alpha_A^{21}=\alpha_A^{22}=0$, so there are no
correlations between a pair of adjacent distinguished edges, or
between an adjacent distinguished and undistinguished edge. We enforce
the mean zero conditions 
\begin{eqnarray*}
\alpha_I^{11} + 2 (N-2) \alpha_A^{11} + \binom{N-2}{2} \alpha_D^{11} +
  (N-2) \alpha_D^{12} = 0 \\
(N-1) \alpha_A^{12} + \binom {N-2}{2} \alpha_D^{12} + \alpha_I^{22} +
  (N-1) \alpha_D^{22} = 0.
\end{eqnarray*}  

Figure
\ref{fig:special} presents the results of some numerical
experiments. The left-most panel depicts a single realization of the
random graph with the undistinguished vertices arranged around a
circle with the distinguished vertex at the center of the circle. The second panel represents the same graph with the vertices
arranged by centrality, with vertices of higher degree closer to the
center.  The final panel gives a histogram of the
vertex degrees of the undistinguished vertices together with a square
(blue online) denoting the sample mean degree of the distinguished
vertex and a diamond (red
online) depicting the sample mean degree of the undistinguished
vertices. 
The first row depicts the case $\alpha_I^{11}=1$ and
$\alpha_I^{22}=1$. In
this case we see a roughly normal distribution of the vertex degrees
around the expected value, $9.9$. If we increase the variance of the
undistinguished edges (keeping the same threshold for all edges) we of
course see the distribution of the degrees of the undistiguished
vertices shift to the right, with no real change in shape. The sample mean degree of the distinguished vertex
remains the same, however.  Next we increase $\alpha_A^{11}$, meaning that we have
positive correlations between undistinguished edges that are incident
to the same vertex (necessarily undistinguished). Here again we see a
distinct broadening of the distribution of the undistinguished vertex degrees, with a
higher probability of having vertices of high and of low degree. One
can see this reflected in the second graph: as compared with the graph
above it there is a denser ``core'' of strongly connected vertices
together with a ``halo'' of weakly connected vertices.  Note that this
change in the degree distribution is very difficult to see in the
circular imbedding of the graph.

\begin{figure}[h]
\centering
\begin{tabular}{|c|c|c|}
\hline
\includegraphics[width=5cm,height=4cm,keepaspectratio]{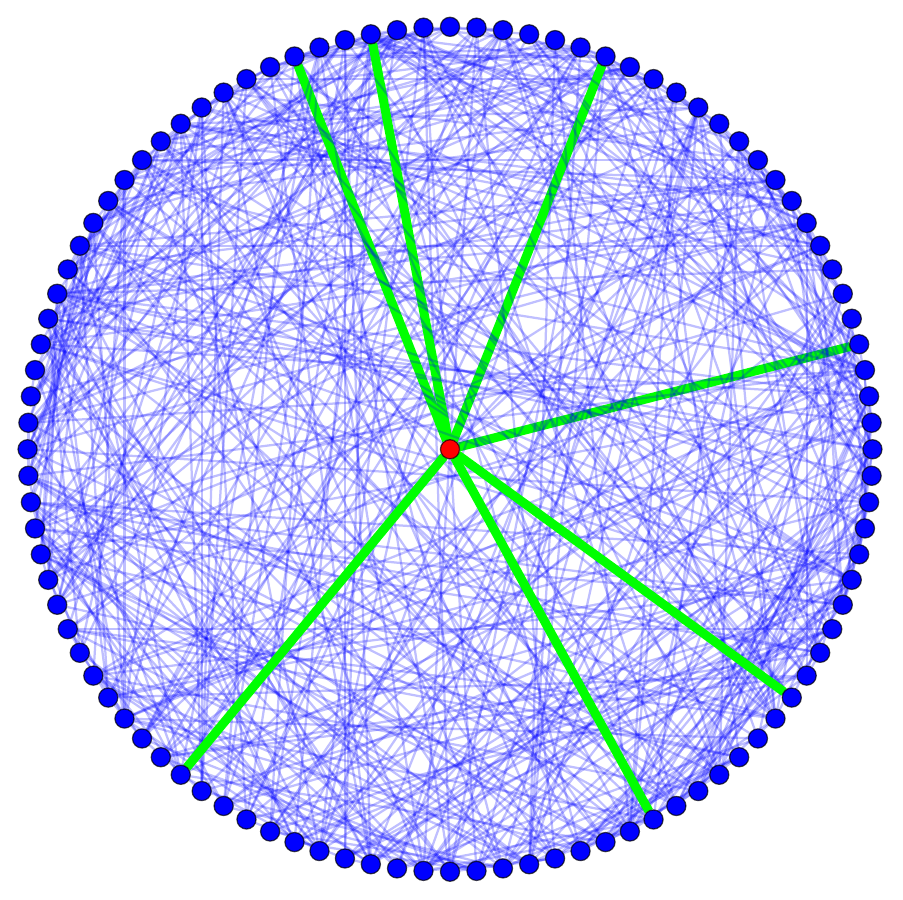} &
\includegraphics[width=5cm,height=4cm,keepaspectratio]{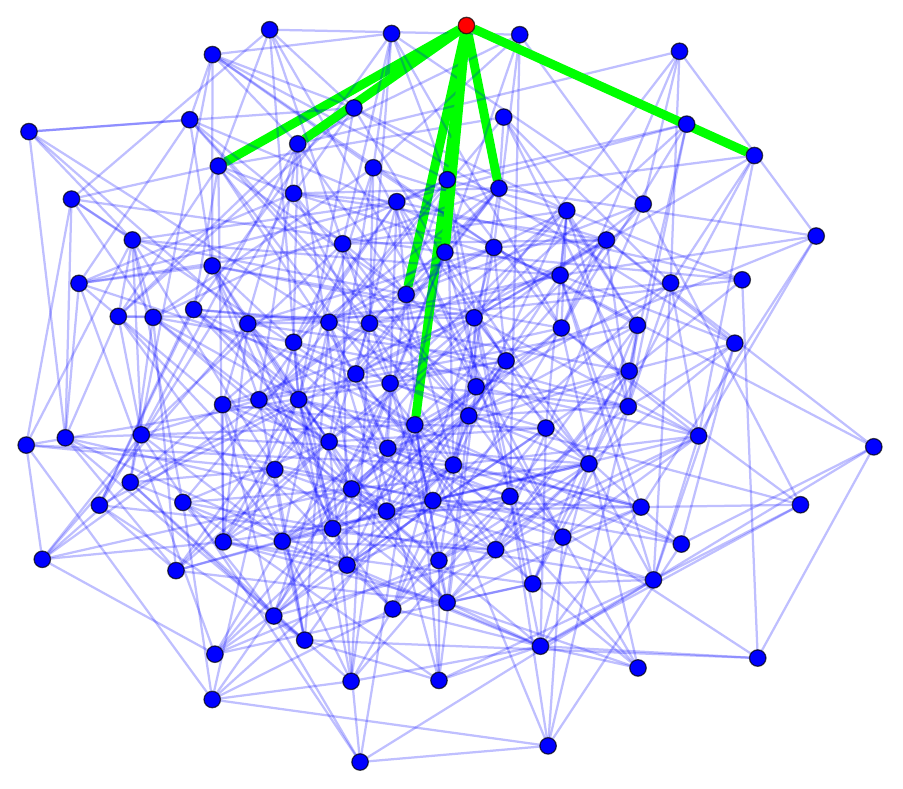} &
\includegraphics[width=5cm,height=4cm,keepaspectratio]{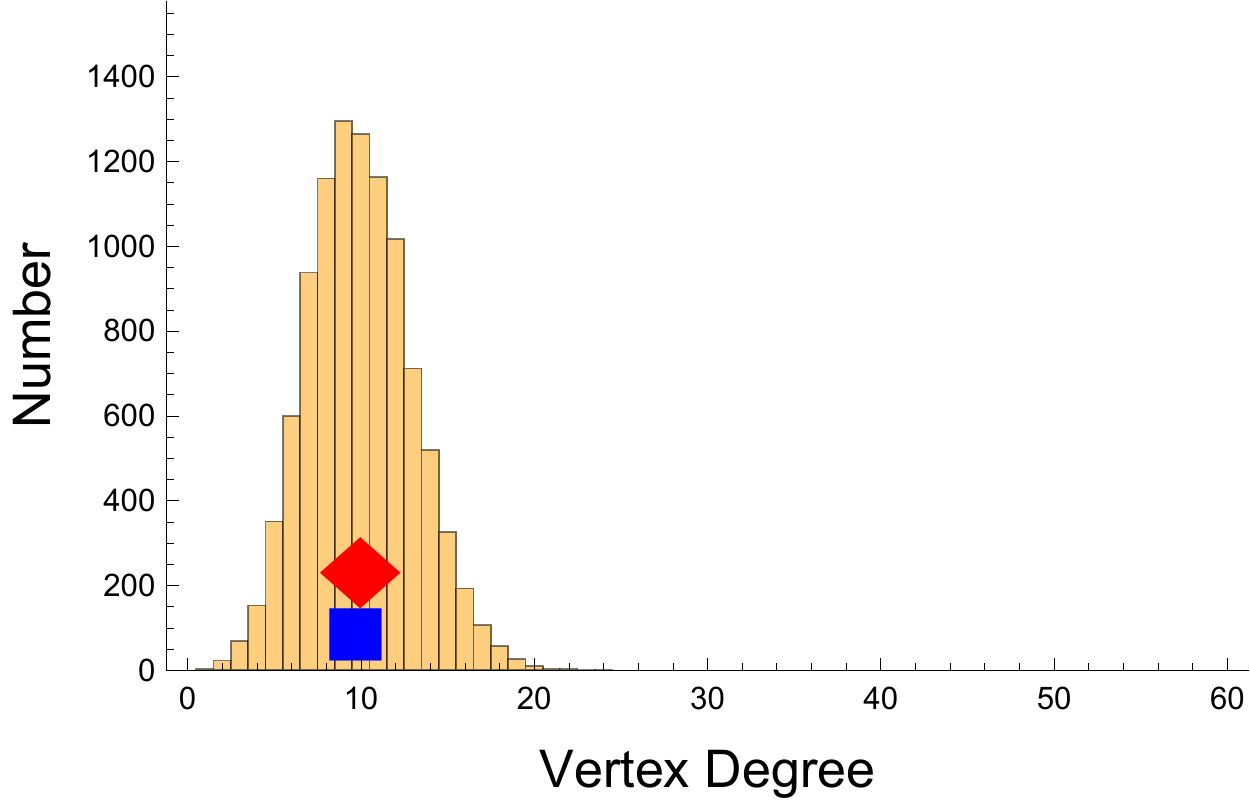}
\\
\hline
\includegraphics[width=5cm,height=4cm,keepaspectratio]{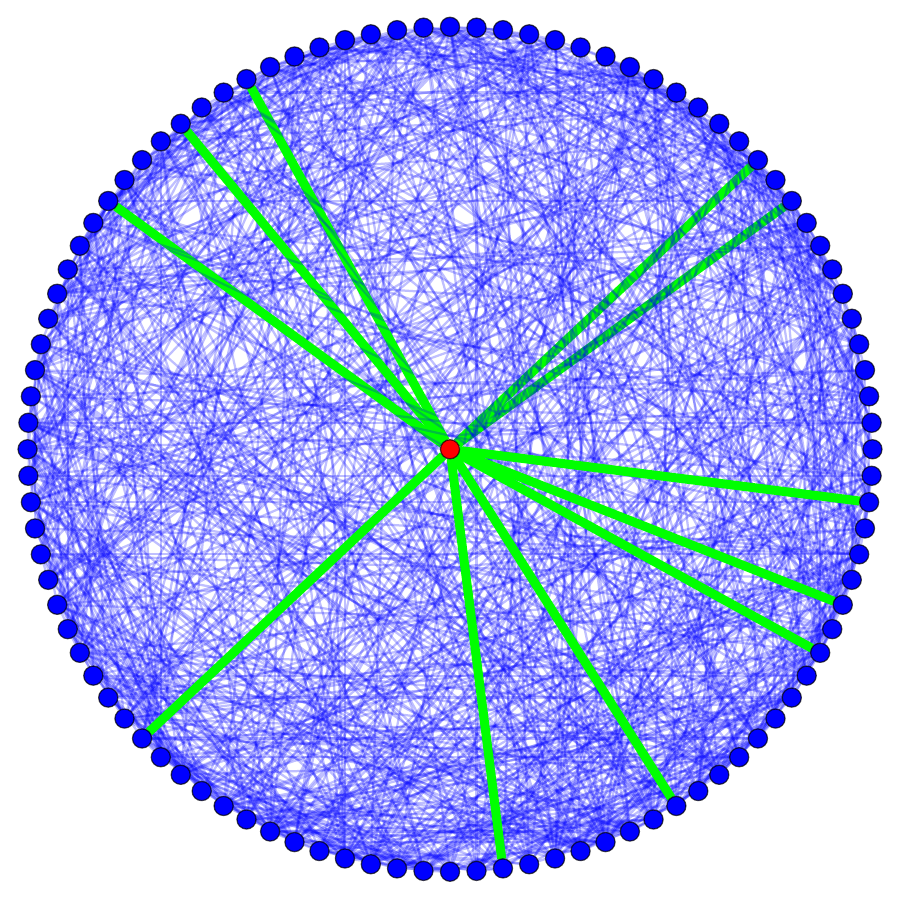} &
\includegraphics[width=5cm,height=4cm,keepaspectratio]{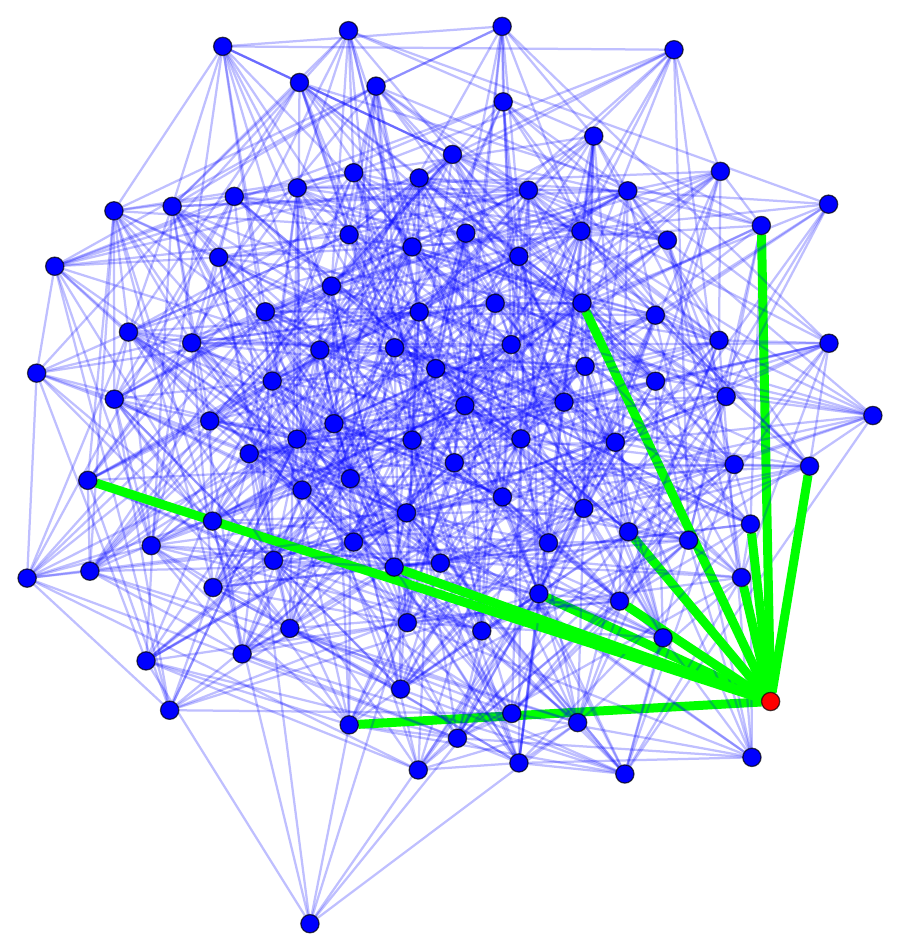} &
\includegraphics[width=5cm,height=4cm,keepaspectratio]{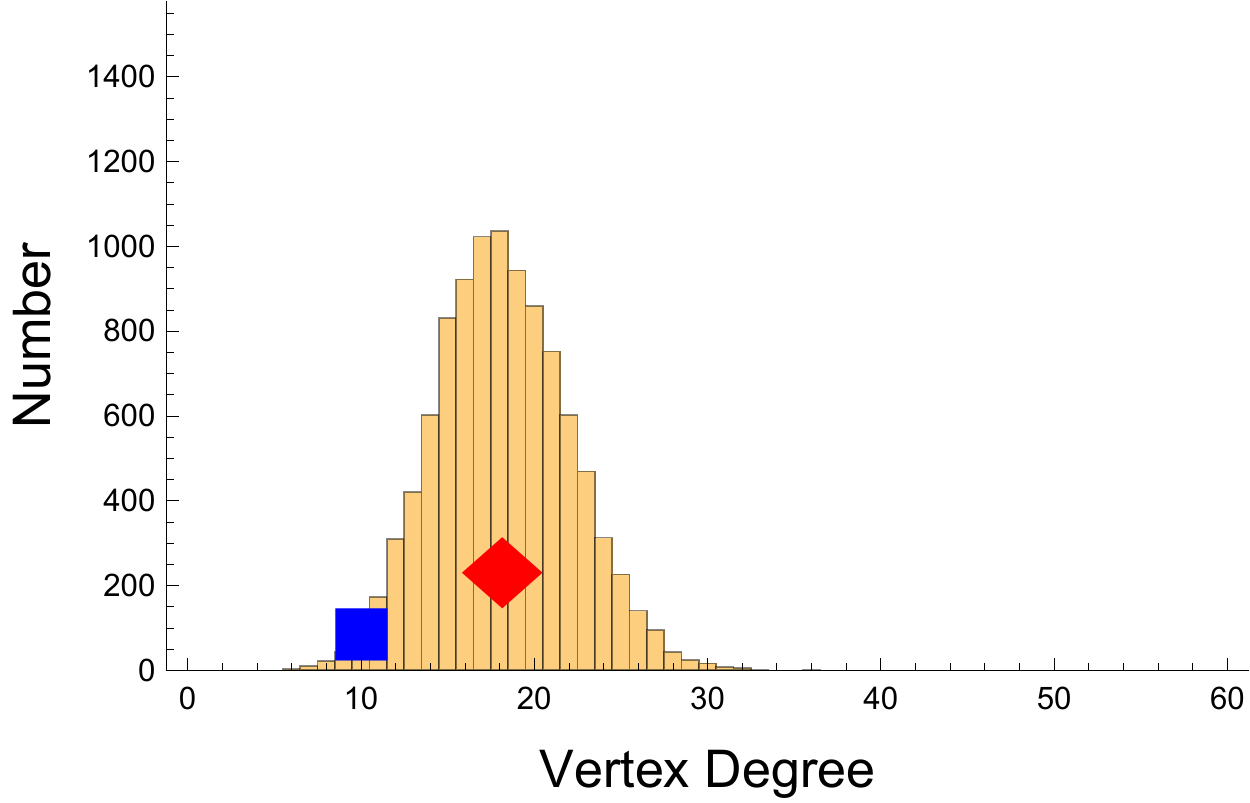}
\\
\hline
\includegraphics[width=5cm,height=4cm,keepaspectratio]{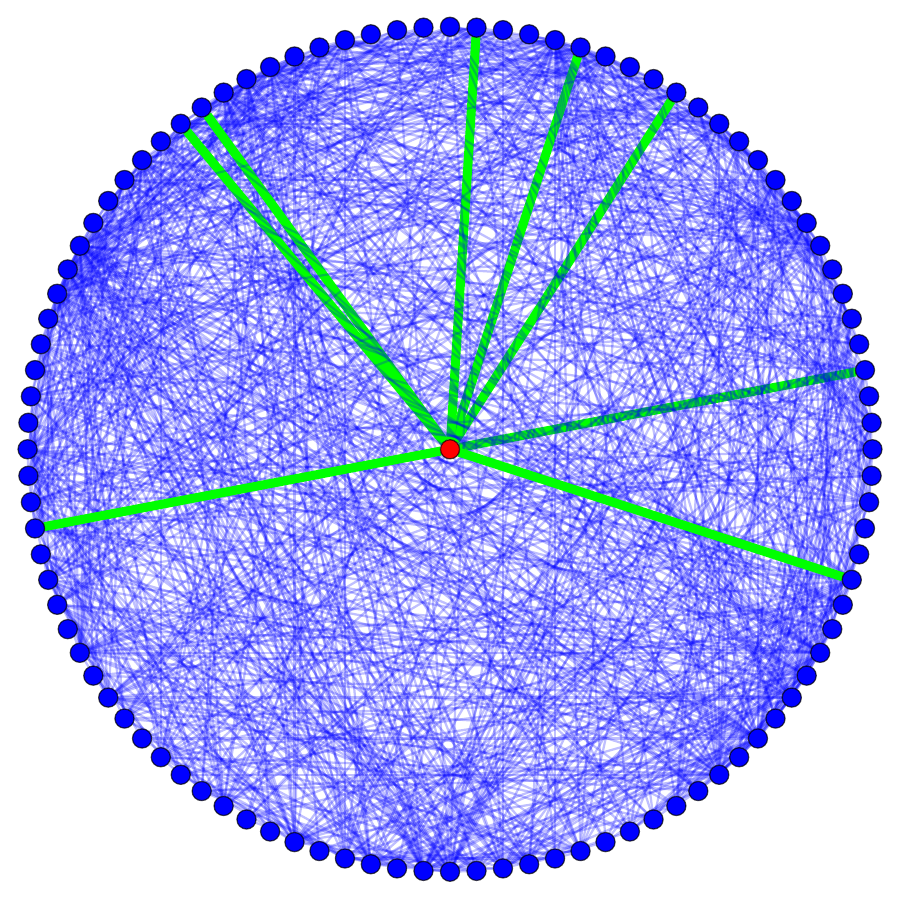} &
\includegraphics[width=5cm,height=4cm,keepaspectratio]{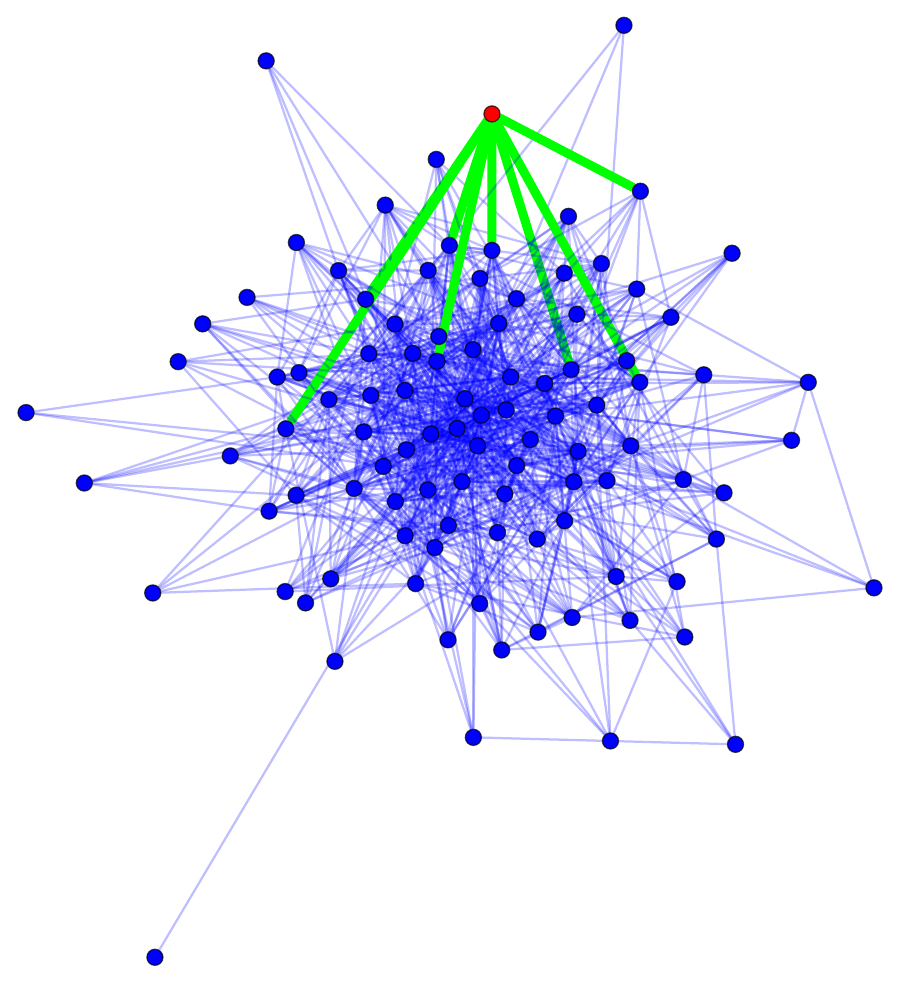} &
\includegraphics[width=5cm,height=4cm,keepaspectratio]{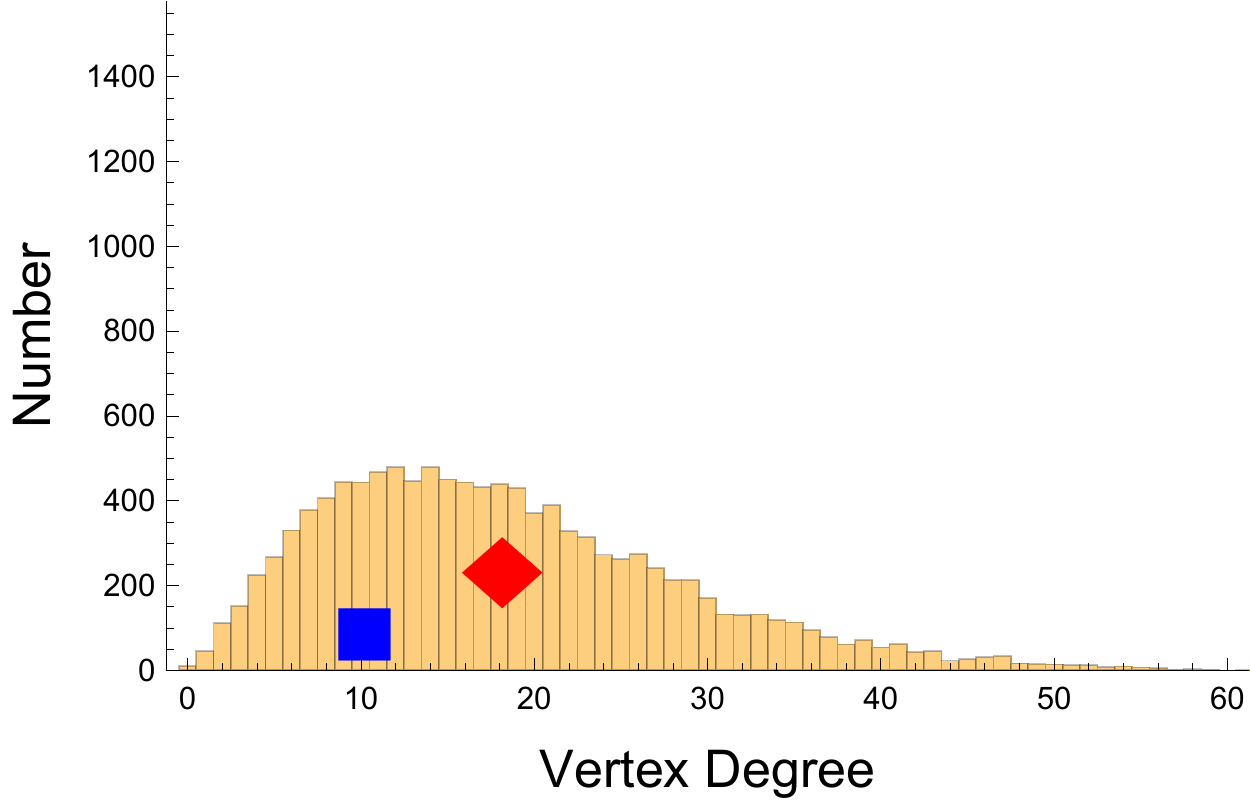}
\\
\hline
Circular Graph & Centrality Graph   &
                                                                     Histogram
                                                                     of
                                                                     Vertex Degrees
\\
\hline
\end{tabular}
\caption{Some numerical experiments for the second order network based
  on a complete graph with one distinguished vertex. The left and
  center panels give different presentations of the same graph. The
  final panel gives a histogram of vertex degrees. }
\label{fig:special}
\end{figure}

This example can itself be generalized in a number of ways. One could consider the case where there are $N$ vertices of one type and $M$ vertices of a second type. In this case there are three types of edge and twenty relations, ten symmetric relations and ten relations that are conjugate-transposes of one another, leading to a covariance matrix depending on fifteen parameters. As in the previous case the actual number of distinct eigenvalues is somewhat less than is guaranteed by the theorem: the theorem guarantees that there are at most twenty distinct eigenvalues, but numerical computations suggest that there are, in fact,  only ten distinct eigenvalues. We have not computed all of the intersection numbers for this case, but it would be relatively straightforward to do so using a symbolic manipulator such as Mathematica. Similarly one could consider a directed analog of the above. 

Another example would be an extension of the SONETS algorithm to generating random subgraphs of a graph other than the complete graph -- say the Johnson $J(N,2)$  graph. In the undirected SONETS based on the complete graph there are only three relations, the identity relation, the adjacent relation, and the disjoint relation. This corresponds to the fact that that line graph of the complete graph (the Johnson graph) is a distance regular graph of diameter two. The line graph of the Johnson graph is {\bf not} a regular graph, but it does have an underlying homogeneous coherent configuration, which can be described as follows. The vertices of the Johnson graph $J(N,2)$ are indexed by two element subsets of $\{1\ldots N\}$. Vertices are connected by an edge if they intersect in one element, so  it is natural to index the edges by a two element subset and a disjoint one element subset, $\{\{i,j\},\{k\}\}$, with $i,j,k$ all distinct, representing the edge from vertex $\{i,k\}$ to vertex $\{j,k\}$.   Given two edges  $\{\{i,j\},\{k\}\}$ and $\{\{i',j'\},\{k'\}\}$ there are a total of twelve relations which are most naturally indexed by a $4$-tuple $(p,q,r,s)$ defined as follows. 
\begin{align*}
p &= |\{i,j\} \cap \{i',j'\}|& \\
q &= |\{k\} \cap \{k'\} | &\\
r &= |\{k\} \cap \{i',j'\}| &\\
s &=   |\{k'\} \cap \{i,j\}|.& 
\end{align*}
Clearly $p\in \{0,1,2\}$ and $q,r,s \in \{0,1\}$, however we have the additional constraint that $r$ and $s$ must be zero if either $p$ or $q$ takes the maximal value ($p=2$ and $q=1$ respectively.) This leads to a homogeneous coherent configuration with twelve relations. One could use this to generate, via the SONETS algorithm, a random subgraph of the Johnson graph in which edges in various relations are correlated.

\section{Conclusions and Future Directions}

We have considered the problem of generating second order networks
random networks -- networks in which the edges are not statistically
independent but rather different two-element
subgraphs (motifs) occur with different probabilities. We show that
the algorithm proposed by Nykamp, Zhao and collaborators for
generating such networks is intimately connected with a certain family
of coherent configurations, and that the underlying algebraic
structure of the coherent configuration, namely the existence of a
fixed dimensional representation of the algebra (the intersection
algebra) makes many of the linear algebraic computations simple to
carry out.  Further we have shown that it is possible to generalize
this structure to generate many new types of random networks. While it
is straightforward to write down (in terms of multidimensional error
type integrals) the probability of having any particular two or,
indeed, $k$ edge motif it is not clear what the large $N$ limiting
distribution of (for instance) quantities like the vertex degree
distribution might be. It would be interesting, though potentially
difficult, to derive asymptotics for the large $N$ limit of quantities
such as the vertex degree distribution. 

{\bf Acknowledgements:} The authors would like to acknowledge support
from the National Science Foundation under Grant DMS-1615418. 

\bibliography{Sonets}
\bibliographystyle{plain}

\section{Appendix}

\begin{proof}[Proof of Proposition \ref{prop:rho}]
We begin by proving that $\rho$ is a homomorphism which is equivalent to demonstrating that
\begin{align*}
\rho(\R^{(i)} \R^{(j)}) = \rho(\R^{(i)}) \rho(\R^{(j)}).
\end{align*}
Expanding the left and right hand sides gives
\begin{align*}
\rho(\R^{(i)} \R^{(j)})_{rs} = \rho( \sum_k \rho(\R^{(i)})_{kj} \R^{(k)})_{rs} = \sum_k \rho(\R^{(i)})_{kj} \rho(\R^{(k)})_{rs}
\end{align*}
and
\begin{align*}
(\rho(\R^{(i)}) \rho(\R^{(j)}))_{rs} = \sum_k \rho(\R^{(i)})_{rk} \rho(\R^{(j)})_{ks}.
\end{align*}
In order to conclude that these two quantities are equal we compute the product $\R^{(i)} \R^{(j)} \R^{(s)}$ in two different ways, namely,
\begin{align*}
(\R^{(i)} \R^{(j)}) \R^{(s)} &= \sum_k \rho(\R^{(i)})_{kj} \R^{(k)} \R^{(s)} = \sum_k \rho(\R^{(i)})_{kj} \sum_r \rho(\R^{(k)})_{rs} \R^{(r)}
\\
&= \sum_r (\sum_k \rho(\R^{(i)})_{kj} \rho(\R^{(k)})_{rs}) \R^{(r)}
\end{align*}
and
\begin{align*}
\R^{(i)} (\R^{(j)} \R^{(s)}) &= \sum_k \rho(\R^{(j)})_{ks} \R^{(i)} \R^{(k)} = \sum_k \rho(\R^{(j)})_{ks} \sum_r \rho(\R^{(i)})_{rk} \R^{(r)}
\\
&= \sum_r (\sum_k \rho(\R^{(i)})_{rk} \rho(\R^{(j)})_{ks}) \R^{(r)}.
\end{align*}
Comparing the coefficients of $\R^{(r)}$ gives the desired result.

Now we turn our attention to eigenvalues. Throughout the rest of the proof we will make use of the simple identity
\begin{align} \label{eq:product}
\M ( \sum_k \beta_k \R^{(k)} ) = \sum_k \bb^\top \rho( \M)_k \R^{(k)}
\end{align}
which follows directly from the definition of $\rho$. It easily follows from \eqref{eq:product} that if $\lambda$ is an eigenvalue of $\rho(\M)$ with left eigenvector $\bb$, then
\begin{align} \label{eq:eigenspace}
\M ( \sum_k \beta_k \R^{(k)}) = \lambda( \sum_k \beta_k \R^{(k)}). 
\end{align}
Since $\R^{(0)}, \R^{(1)}, \ldots,\R^{(d)}$ have disjoint support we know that $\rank( \sum_k \beta_k \R^{(k)}) \ge 1$ so that $\lambda$ is also an eigenvector of $\M$. To show the converse, namely, that if $\lambda$ is an eigenvalue of $\M$, then it is an eigenvalue of $\rho(\M)$, we note that it suffices to only consider the case $\lambda = 0$. This is because $\rho$ is linear, $\I$ is in the algebra generated by $\R^{(0)}, \R^{(1)}, \ldots,\R^{(d)}$, and $\rho(\I)$ is the identity matrix in its algebra. Therefore for the sake of contradiction suppose that zero is an eigenvalue of $\M$ but not of $\rho(\M)$. By \eqref{eq:product} we see that
\begin{align*}
\det(\sum_k \bb^\top \rho( \M)_k \R^{(k)}) = 0
\end{align*}
for any $\bb$. Since zero is not an eigenvalue of $\rho(\M)$ the map $\bb \mapsto \bb^\top \rho(\M)$ is surjective and hence there exists a $\bb$ for which $\sum_k \bb^\top \rho( \M)_k \R^{(k)} = \I$. This of course leads to the contradiction $\det(\I)=0$ completing the argument.

Next we prove that
\begin{align*}
\mult(\lambda) = \rank ((\sum_k \beta_k \R^{(k)})_{\bb \in \mathcal{B}_\lambda}).
\end{align*}
We will start by proving that $\mult(\lambda)$ is at least as large as the right hand side. As mentioned before it suffices to consider $\lambda = 0$. To start we note that there exists a $k \ge 1$ for which $\mathcal{B}_\lambda$ is a basis for the left null space of $\rho(\M)^k = \rho(\M^k)$. By \eqref{eq:eigenspace} this implies that the dimension of the left eigenspace of $\M^k$ is at least $\rank ((\sum_k \beta_k \R^{(k)})_{\bb \in \mathcal{B}_\lambda})$. However, every eigenvector of $\M^k$ is a generalized eigenvector of $\M$. This establishes our inequality. Finally, there exists a vector $\bg$ so that
\begin{align*}
\sum_{\lambda \in \sigma(\rho(\M))} \mult(\lambda) & \ge \sum_{\lambda \in \sigma(\rho(\M))} \rank((\sum_k \beta_k \R^{(k)})_{\bb \in \mathcal{B}_\lambda})
\\
& \ge \sum_{\lambda \in \sigma(\rho(\M))} \rank( \sum_k \sum_{\bb \in \mathcal{B}_\lambda} \beta_k \bg_\bb R^{(i)})
\\
& \ge \rank( \sum_k \sum_{\lambda \in \sigma(\rho(\M))} \sum_{\bb \in \mathcal{B}_\lambda} \beta_k \bg_\bb \R^{(k)})
\\
&= \rank(\I).
\end{align*}
Therefore all inequalities are in fact equalities giving the result.

Now to show that $\rho$ is injective suppose that $\rho(\M) = \0$. Then $\rho(\M)$ has a left eigenbasis for the eigenvalue zero. By repeating the above argument with $k=1$ we see that $\M$ has a right eigenbasis for the eigenvalue zero and hence is zero. This completes the proof.

Finally we consider diagonalizability. If $P$ is a polynomial then clearly $\rho(P(\M))=P(\rho(\M))$. In particular, if $P$ is the minimal polynomial of $\M$ then $P(\rho(\M))=0$. A necessary and sufficient condition for diagonalizability is that the eigenvalues are simple roots of the minimal polynomial. Since $\M$ and $\rho(\M)$ have the same eigenvalues the result follows.  
\end{proof}

\end{document}